\documentclass[11pt, a4paper]{amsart}
\usepackage{amssymb,amsmath,amscd,amsfonts,mathrsfs,yfonts}
\usepackage{faktor}
\usepackage{textcomp}
\usepackage{pgfplots}
\usepackage{ stmaryrd }
\pgfplotsset{height=5cm,width=8cm,compat=newest}
\usepackage{appendix}
\usepackage{upgreek}
\usepackage[matrix,arrow,curve,pic,2cell]{xy}
\UseTwocells
\xyoption{2cell}
\usepackage{tkz-graph}
\usepackage[left=3cm,right=3cm,top=2.5cm,bottom=2.5cm]{geometry}

\usetikzlibrary{decorations.markings}
\usetikzlibrary{shapes.geometric}
\usetikzlibrary{patterns}
\newtheorem{definition}{Definition}
\newtheorem{theorem}{Theorem}[section]
\newtheorem{proposition}{Proposition}[section]
\newtheorem{lemma}{Lemma}[section]
\newtheorem{remark}{Remark}
\newtheorem{corollary}{Corollary}
\newenvironment{Sproof}{\paragraph{\it Proof sketch.}}{ \hfill$\square$ \vskip 10pt}

\DeclareMathOperator{\im}{im}

 \DeclareMathOperator{\Diff}{\mathfrak{Diff}}

\DeclareMathOperator{\supp}{supp}
\DeclareMathOperator{\vol}{Vol}

\author[V. Gol{\textquotesingle}dshtein]{Vladimir Gol{\textquotesingle}dshtein}
 \address{V. Gol{\textquotesingle}dshtein: Department of Mathematics, Ben Gurion University of the Negev, P. O. Box 653, Beer Sheva, Israel}
 \email{vladimir@bgu.ac.il}
\author[R. Panenko]{Roman Panenko} 
\address{R. Panenko: Department of Mathematics, Ben Gurion University of the Negev, P. O. Box 653, Beer Sheva, Israel}
\email{panenkora@gmail.com}
\title{On the de Rham homomorphism for  $L_\pi$-cohomologies}
\begin{document}
\maketitle

\begin{abstract}
We study the procedure of regularization in the context of the Lipschitz version of de Rham calculus  on metric  simplicial complexes with bounded geometry. 
It provides us with the machinery to handle the de Rham homomorphism for  $L_\pi$-cohomologies. 
In this respect, we obtain the condition resolving the question of triviality of the kernel for de Rham homomorphism. In particular, 
we specify the  non-trivial cohomology classes explicitly  for a sequence of parameters $\pi = \langle p_0,\,p_1,\dots,\,p_n\rangle$ missing non-increasing monotonicity.

\vspace{2mm}

\vspace{2mm}
\noindent
\textbf{Mathematic Subject Classification 2000: } 58A12, 58A15, 53C65, 57Q05, 57R12, 51F30, 46E30
\end{abstract}
\tableofcontents
\section{Introduction}
\subsection{Notation}
We will use the following notation.  
\begin{itemize}
\item[---] 
$K$~is a metric simplicial complex of bounded geometry, $K[m]$ is an $m$-skeleton of $K$ on the assumption that there exists  an isometric embedding $K[m] \hookrightarrow K$
\vskip 10pt
\item[---]
$
\pi = \langle p_0,\,p_1,\dots,\,p_n\rangle\subset (1,\,\infty) ,~\frac{1}{p_{i+1}}-\frac{1}{p_i}\le\frac{1}{n}
$
\vskip 10pt
\item[---]
$
\Omega_{\pi}(K) = \Big(\Omega^i_{p_i,\,p_{i+1}}(K),\,d^{i}\Big)
$
is a Sobolev-de Rham complex
\vskip 10pt
\item[---]
$
 \Omega_{ \text{\fontsize{5}{6}\selectfont {\rm dRh}}}(U) = \Big( \Omega^{i}_{ \text{\fontsize{5}{6}\selectfont {\rm dRh}}}(U),\,d^{i}\Big)
$
is a de Rham complex of $C^\infty$-forms on $U\subset \mathbb R^n$
\vskip 10pt
\item[---]
$
S\mathscr L_{\pi}(K) = \Big(S \mathscr L^i_{p_i,\,p_{i+1}}(K),\,d^{i}\Big)
$
is a Sobolev-de Rham complex consisting of  smooth on any simplex Lipschitz forms on $K$
\vskip 10pt
\item[---]
$
C_{\pi}(K) = \Big(C^i_{p_i,\,p_{i+1}}(K),\,\partial^{i}\Big)
$
is a  complex of  Sobolev spaces of simplicial cochains
\vskip 10pt
\item[---]
$\mathscr R$~is the de Rham regularization  operator on $\Omega_{\pi}(K)$
\vskip 10pt
\item[---]
$\mathscr A$~is the de Rham homotopy  operator  $\mathscr R \Rightarrow 1_{\Omega_{\pi}}$
\vskip 10pt
\item[---]
$\mathscr I$~is the de Rham type homomorphism  $S{\mathscr L}_{\pi} \rightarrow C_{\pi}$
\vskip 10pt
\item[---]
${\sf Ban}$~is the category of Banach spaces with continuous (bounded) maps as morphisms.  
\end{itemize}
\subsection{Brief summary}
The procedure of regularization introduced by de Rham borrows and generalizes the standard approach  which was elaborated in the theory of Schwartz  distributions. 
Let $\mathscr D^{\prime}(\mathbb R^n)$ be a  space of continuous functionals on the space of compactly supported smooth functions $C^{\infty}_{0}(\mathbb R^n)$ 
endowed with the usual topology.
Let  $T\ast \varphi$ be the convolution of a distribution $T \in \mathscr D^{\prime}(\mathbb R^n)$ and a function $\varphi(y) \in C^{\infty}_{0}(\mathbb R^n)$.
There is a well-known fact (see, for example,  \cite{WR}):
\begin{theorem}
Let  $\varphi_{\varepsilon} \in C^{\infty}_{0}(\mathbb R^n)$ be a sequence of positive functions such that
$$
\int_{\mathbb R^n} \varphi_{\varepsilon}(x)dx  = 1
$$
and $\rm{supp}(\varphi_{\epsilon})$ is a ball with radius $\varepsilon$.
If  $T \in \mathscr D^{\prime}(\mathbb R^n)$  it follows that
$T\ast \varphi_{\varepsilon}  \in C^{\infty}(\mathbb R^n)$ and that  $T\ast \varphi_{\varepsilon} \to T$ in $\mathscr D^{\prime}(\mathbb R^n)$ as $\varepsilon \to 0$.
\end{theorem}
To sum up, the function $\varphi_{\varepsilon}$ specifies an operator $\Phi_\varepsilon$ such that 
$$
\Phi_\varepsilon(T) = T\ast \varphi_{\varepsilon}.
$$
If $\varphi_{\varepsilon}$ is a symmetric kernel, that is $\varphi_{\varepsilon}(x) = \varphi_{\varepsilon}(-x)$, and $g  \in C^{\infty}_{0}(\mathbb R^n)$ 
we have
$
\langle T\ast \varphi_{\varepsilon},\, g\rangle = \langle T,\, g \ast \varphi_{\varepsilon}\rangle. 
$
Then we can define a regularization operator with a symmetric kernel on the space of distributions as follows
$
\{\Phi_{\varepsilon}(T)\} (g) = T(\Phi_{\varepsilon}(g)).
$

The de Rham's construction is adapted for the space of continuous functionals on the space of compactly supported smooth differential forms. 
And so we want our operator $R$ to be of the following form
$$
R T[\omega] = T[R^*\omega].
$$
Following the approach presented in the work \cite{GKS1}, we restrict our attention to a special kind of currents, which can be presented as elements
of the Sobolev space of differential forms $\Omega^k_{p_k,\,p_{k+1} }(M)$.  

Suppose that ${\bf B}_1 \subset \mathbb R^n$ is an open ball. For   $\frac{1}{p}-\frac{1}{q}<\frac{1}{n}$ and  $\frac{1}{r}-\frac{1}{p}<\frac{1}{n}$ there
exists an operator $S\colon \Omega^{i}_{ \text{\fontsize{5}{6}\selectfont {\rm dRh}}}({\bf B}_1) \to \Omega^{i-1}_{ \text{\fontsize{5}{6}\selectfont {\rm dRh}}}({\bf B}_1) $
such that (see \cite{GT}) 
$$
\omega = Sd\omega+dS\omega 
$$
and, moreover, in case $\omega \in \Omega^i_{p,\,r}({\bf B}_1) \cap  \Omega^{i}_{ \text{\fontsize{5}{6}\selectfont {\rm dRh}}}({\bf B}_1)$, we  obtain  
$$
\|S\omega\|_{\Omega^{i-1}_{q,\,p}(B)} \le C\| \omega\|_{\Omega^i_{p,\,r}(B)}
$$

Assume that $\tau$  is a compactly supported smooth $n$-form such that $\tau(-x) = -\tau(x)$ and
$$
\int_{\mathbb R^n}\tau = 1,
$$
then we can define operators
$$
{R} \omega= \int\limits_{{\mathbb R}^n}\mathfrak s^{*}_{\varepsilon v} \omega(x) \cdot \tau(v),
$$
where $\varepsilon \in [0,\,1]$ and $\mathfrak s_{v}$ is a special family of diffeomorphisms parametrized  by
$$
\mathbb R^n \to \Diff(B),
$$
$$
v \mapsto \mathfrak s_{v}.
$$
Then we can define 
$$
A  = (R - \mathbf{1})\circ S.
$$
that implies 
$$
R - \mathbf{1} = dA + Ad.
$$
Then we can extend it to the space of currents
$$
R T[\omega] = T[R^*\omega]
$$
and
\begin{align*}
RT[\omega] &= T[R^*\omega] = T[A^*d\omega + dA^*\omega] =  T[A^*d\omega] + T[dA^*\omega] \\
&= AT[d\omega]+ \partial T[A^*\omega] = \{\partial A T + A \partial T \}[\omega]. 
\end{align*}
As a special case of those constructions, it holds for Sobolev spaces of differential forms.

Suppose there exists a bi-Lipschitz homeomorphism $\varphi\colon U \to {\bf B}_1$. Hence, we can define a couple of operators inheriting  the mentioned  properties 
$$
\Big\{\varphi^*R(\varphi^{-1})^*\Big\}\colon \Omega^i_{p,\,r}(U) \to \Omega^i_{p,\,r}(U)
$$
and
$$
\Big\{\varphi^*A(\varphi^{-1})^* \Big\}\colon \Omega^i_{p,\,r}(U) \to \Omega^{i-1}_{q,\,p}(U).
$$

Then, for a Lipschitz atlas, that construction can be globalized by gluing locally acting operators on charts. 

In this respect, the main part of the current work is devoted to examining the basic properties of the de Rham regularization operator $\mathscr R$. 
We apply the procedure of regularization in the context of metric  simplicial complexes with bounded geometry (see, for example, \cite{DK}):

\begin{definition}
We will say that a  metric simplicial complex $K$ is a star-bounded complex if there exists  $N < \infty $ such that 
every vertex of $K$ is incident to at most $N$ edges.
\end{definition}
\begin{definition}
A metric simplicial complex $K$ has bounded geometry if it is connected, star-bounded, and there exists $L\ge 1$ such that the length of every edge is in the interval $[L^{-1},\,L]$.
\end{definition}
\begin{remark}
Following \cite{DK}, the notion of bounded geometry includes star-boundedness. We highlight this property in order to emphasize  the connection  with the work \cite{GKS2}.
\end{remark}
\begin{remark}
Below, we will assume that all complexes have bounded geometry with $L = 1$.
\end{remark}
Those results  have some useful applications regarding an isomorphism $H^i(\Omega_\pi) \cong H^i(C_\pi)$ and the Sobolev inequality. 
In particular, the obtained results allow one to transfer the following assertion \cite{GT}:
\begin{theorem} 
$H^i(\Omega_{\pi}) = 0$ if and only if for all $\omega \in \ker d^i$, there exists  a differential form $\theta$ of degree $i-1$ such that $d\theta = \omega$ and
$$
\|\theta\|_{\Omega^{i-1}_{\pi}}\le C \|\omega\|_{\Omega^i_{\pi}}.
$$
\end{theorem}
\noindent to the case of metric simplicial complexes with bounded geometry.

It is enlightening to note that,
generally, $H^i(\Omega_\pi) \to H^i(C_\pi)$ is a homomorphism and its kernel vanishes only for proper $\pi$. We specify a class of differential forms, which can make 
this kernel nontrivial [{\bf Theorem \ref{NonTriv}}]. 

It is also worth mentioning that the  theory of $L_\pi$-cohomology for Riemannian manifolds of bounded geometry with a triangulation  gets thoroughly explored  in  Ducret's thesis \cite{Duc}.
In this respect, it is relevant to note that a bi-Lipschitz triangulation of Riemannian manifolds with bounded geometry is provided by results  \cite{OA}. 
{\centering
\section{Hitchhiker's Guide to The Results}
}
This section should give an informal overview of the  basic results obtained in the present work. 
The explanation below is devoted to our intention to reveal the  essential aspects lying behind the forthcoming tons of calculations. 
So it can serve as a hitchhiker's guide to the presented article.  
Let us draw  attention to the category that consists of complexes of Sobolev spaces and bounded maps between them. 
For our purposes, we primarily focus on  Sobolev-de Rham complexes and complexes of Sobolev spaces of simplicial cochains. 
Following de Rham, we introduce a map [{\bf Section \ref{DefReg}}]
$$
\xymatrix{\Omega_{\pi}\ar[r]^{\mathscr R}&\Omega_{\pi}}
$$
that enjoys  [{\bf Section \ref{Decomp}}] the property to be factored  through a nicely defined  complex in the following manner:  
$$
\xymatrix@C = 0.2em{
&S{\mathscr L}_{\pi}\ar@{->>}[rd]^{\epsilon}&\\
\Omega_{\pi}\ar[rr]^{\mathscr R}\ar[ru]^{\rho}&&\Omega_{\pi}
}
$$
Another benefit is  the following homotopy [{\bf Lemma \ref{Hm}}]
$$
\xymatrix{\Omega_{\pi} \rtwocell<8>^{<1.5>{\mathscr R}}_{<1.5>{1_{\Omega_{\pi}}}}{\  \mathscr A}& \Omega_{\pi}}.
$$
Combining the preceding couple of diagrams we obtain 
$$
\epsilon \rho \sim 1
$$  
And as a result, $\rho$ is a quasi-isomorphism. 

Turning to the de Rham map between the complex of differential forms and the complex of  simplicial cochains
$$
\xymatrix{S{\mathscr L}_{\pi}\ar[r]^{\mathscr I}&C_{\pi}},
$$
it can be proved that our map is a split epimorphism  [{\bf Section \ref{dRhOp}}]:
$$
S{\mathscr L}_{\pi}\cong C_{\pi}\oplus \ker_{\pi}{\mathscr I}
$$
In contrast with the classical case, there arises  the question of the triviality of the kernel.   
The theorem below shows us that, whenever we impose a proper condition on $\pi$,   there exists  an isomorphism between the following  complexes up to homotopy: 
$$
 \ker_{\pi}{\mathscr I}\cong 0.
$$
\begin{theorem}
Given a non-increasing sequence $\pi = \langle p_0,\,p_1,\dots,\,p_n\rangle$, then $ \ker_{\pi}{\mathscr I} $ is a contractible complex 
$$
\xymatrix{ \ker_{\pi}{\mathscr I} \rtwocell<8>^{<1.5>{1}}_{<1.5>{0}}{\  }& \ker_{\pi}{\mathscr I} }.
$$
\end{theorem} 
\begin{Sproof}
Indeed, let us take a look at the complex
$$
 \ker_{\pi}{\mathscr I} = ( \ker^i_{\pi}{\mathscr I},\,d^{i}),
$$
we will use the notation $Z^i_\pi$ for $\ker d^i$.
Due to   [{\bf Proposition \ref{TrivKer}}], the monotonicity of $\pi = \langle p_0,\,p_1,\dots,\,p_n\rangle$ and $H^{i}( \ker_{\pi}{\mathscr I}) = 0$ guarantee us that there exists a family of bounded maps
$$
\eta^i \colon Z^i_\pi  \to  \ker^{i-1}_{\pi}{\mathscr I}
$$
satisfying 
$
d^{i-1}\eta^{i} = 1.
$
For 
$
H^{n}( \ker_{\pi}{\mathscr I}) = 0
$
we have
$
d^n = 0.
$
It follows that for $h^n = \eta^n$
$$h^{n}\colon  \ker^n_{\pi}{\mathscr I} \to  \ker^{n-1}_{\pi}{\mathscr I}$$
we have
$
d^{n-1}h^{n} = 1.
$
Assume that $\alpha^{n-1} = 1- h^n d^{n-1}$, then we have
$$
d^{n-1}\alpha^{n-1} = d^{n-1}(1-h^{n}d^{n-1}) = d^{n-1} - (d^{n-1}h^{n})d^{n-1} = 0.
$$
As a result, we obtain that $\alpha^{n-1}\big( \ker^{n-1}_{\pi}{\mathscr I}\big) \subset Z_{\pi}^{n-1}$.  Consider
$$\eta^{n-1}\colon  Z^{n-1}_{\pi} \to  \ker^{n-2}_{\pi}{\mathscr I}$$ 
for which
$d^{n-2}\eta^{n-1} = 1$. Let us put $h^{n-1} = \eta^{n-1}\alpha^{n-1}$
$$
\xymatrix{
&\ker^{n-1}_{\pi}{\mathscr I}\ar[d]^{\alpha^{n-1}}\ar@/{}^{1.5pc}/[rd]^{0} \ar@/{}_{2pc}/[ld]_{h^{n-1}}&\\
\ker^{n-2}_{\pi}{\mathscr I}\ar@/{}_{1.2pc}/[r]_{d^{n-2}}&Z^{n-1}_{\pi}\ar[r]_{d^{n-1}} \ar@/{}_{1.2pc}/[l]_{\eta^{n-1}}&0
}
$$
That implies 
$$
d^{n-2}h^{n-1} = (d^{n-2}\eta^{n-1})\alpha^{n-1}  =  \alpha^{n-1}
$$
and
$$
1 = d^{n-2}h^{n-1}+ h^n d^{n-1}.
$$
Assume that we have $h^n,\,h^{n-1},\dots,\,h^{i}$ for which there holds the equation 
$$
1 = d^{j-1}h^{j}+ h^{j+1} d^{j}
$$
where $j\ge i$.
 Let us put $\alpha^{i-1} = 1- h^i d^{i-1}$
 $$
d^{i-1} \alpha^{i-1} = d^{i-1} - (d^{i-1}h^{i})d^{i-1} = d^{i-1} - (1-h^{i+1}d^{i})d^{i-1} = 0.
 $$
As a result, it holds that $\alpha^{i-1}\big( \ker^{i-1}_{\pi}{\mathscr I}\big) \subset Z_{\pi}^{i-1}$.
Moreover, there exists 
$$\eta^{i-1}\colon  Z^{i-1}_{\pi} \to  \ker^{i-2}_{\pi}{\mathscr I}$$ 
such that
$d^{i-2}\eta^{i-1} = 1$. Let us put $h^{i-1} = \eta^{i-1}\alpha^{i-1}$
$$
\xymatrix{
&\ker^{i-1}_{\pi}{\mathscr I}\ar[d]^{\alpha^{i-1}}\ar@/{}^{1.5pc}/[rd]^{0} \ar@/{}_{2pc}/[ld]_{h^{i-1}}&\\
\ker^{i-2}_{\pi}{\mathscr I}\ar@/{}_{1.2pc}/[r]_{d^{i-2}}&Z^{i-1}_{\pi}\ar[r]_{d^{i-1}} \ar@/{}_{1.2pc}/[l]_{\eta^{i-1}}&0
}
$$
That implies 
$$
d^{i-2}h^{i-1} = (d^{i-2}\eta^{i-1})\alpha^{i-1}  =  \alpha^{i-1}
$$
and
$$
1 = d^{i-2}h^{i-1}+ h^i d^{i-1}.
$$
Repeating that argument, we can finalize our proof inductively. 
So we obtain the homotopy equivalence
$$
1_{ \ker_{\pi}{\mathscr I}} \sim 0.
$$
\end{Sproof}
\begin{corollary}
Given a non-increasing sequence $\pi = \langle p_0,\,p_1,\dots,\,p_n\rangle$, then
$$
 \ker_{\pi}{\mathscr I}\cong 0.
$$
\end{corollary}
\begin{Sproof}
We have a couple of trivial maps $\alpha = 0$ and $\beta = (0\hookrightarrow \ker_{\pi}{\mathscr I})$:
$$
\xymatrix{
\ker_{\pi}{\mathscr I}\ar[r]^{\alpha}\ar@/{}^{2pc}/[rr]^{\beta\alpha = 0}\ar@/{}_{2pc}/[rr]_{1}&0\ar[r]^{\beta} &\ker_{\pi}{\mathscr I}
}
$$
Thanks to the previous theorem, we have
$$
\xymatrix{ \ker_{\pi}{\mathscr I} \rtwocell<8>^{<1.5>{1}}_{<1.5>{0}}{\  h}& \ker_{\pi}{\mathscr I} }.
$$
That implies
$$
1 - \beta\alpha  = dh+hd
$$
and it follows that 
$$
\beta\alpha \sim 1
$$
and 
$\alpha$ is a quasi-isomorphism.
\end{Sproof}
By contrast with the previous, assume that $\pi$ misses monotonicity at the $i$-th term.     
 Whenever our complex is contractible 
 $$
\xymatrix{ \ker_{\pi}{\mathscr I} \rtwocell<8>^{<1.5>{1}}_{<1.5>{0}}{\  }& \ker_{\pi}{\mathscr I} }
$$
for any subobject $A$, it also holds true that
\vskip 10pt
$$
\xymatrix{ \ker_{\pi}{\mathscr I}\ar@/{}_{3pc}/[rr]_{0}\ar@/{}^{3pc}/[rr]^{1}\ar@{<<-}[r]&A \rtwocell<5>^{<1.3>{}}_{<1.3>{0}}{\  }& ~~~\ker_{\pi}{\mathscr I} }.
$$
\vskip 10pt
where the triangles are commutative. Below  [{\bf Theorem \ref{NonTriv}}]  we specify such a subobject $dN$:
$$
\xymatrix{
0\ar[r]&\dots\ar[r]&0\ar[r]&\underset{i}{d(N)}\ar[r]&0\ar[r]&\dots\ar[r]&0
}
$$
for which the last diagram fails. So such a subobject could serve as a kind of obstruction to the contractibility. And as a result, $\ker \big(S{\mathscr L}_{\pi} \xrightarrow{\mathscr I}C_{\pi}\big)$ is not trivial.

Thanks to the previous, it follows that
\begin{remark}
Assume that  a sequence $\pi = \langle p_0,\,p_1,\dots,\,p_n\rangle$ misses monotonicity at the $i$-th term, then 
$$
H^i(\Omega_\pi) \neq 0.
$$
\end{remark}
Indeed, loosely speaking, $H^i$ is an additive functor which  could be  understood as defined on the  homotopy category of chain complexes. 
Since
$$
\Omega_\pi \cong S{\mathscr L}_{\pi}\cong C_{\pi}\oplus \ker_{\pi}{\mathscr I}
$$
up to homotopy, the term
$$
 \ker_{\pi}{\mathscr I}  \not\cong 0
$$
implies 
$$
\Omega_\pi \not\cong 0
$$
in  the  homotopy category of chain complexes. In  [{\bf Theorem \ref{NonTriv}}] we specify non-trivial cohomology classes explicitly. 
\begin{remark}
That construction implies the following effect.  For a Sobolev-de Rham complex with  such  a non-monotonic $\pi$,  the only possibility
of having zero cohomology is the absence of bounded geometry.
\end{remark}

{\centering
\section{Sobolev Spaces}

}
Let $K$ be a metric $n$-dimensional simplicial complex and $K[n]$ be a set of simplices of the maximal dimension $n$.  
\begin{definition}
We will use the following notation for $L_p$-norms of differential forms
$$\|\omega\|_{\Omega^k_p(K)} = \Bigg\{\sum_{T\in K([n])} \int\limits_T |\omega(x)|^p dx\Bigg\}^{\frac{1}{p}}$$
\end{definition}
\begin{remark}
In the case $p = \infty$  we put 
$$\|\omega\|_{\Omega^k_{\infty}(T)} = {\rm ess}\sup_{x\in T}|\omega(x)|$$
\end{remark}

\begin{theorem}\emph{(Rademacher's theorem)}
Let $U \subset {\mathbb R}^n$, and let ${f \colon U \to {\mathbb R}}$ be locally Lipschitz. Then $f$ is differentiable at almost every point in $U$.
\end{theorem} 
Due the Rademacher's theorem, for every Lipschitz function $f$, 
we can define $df= \frac{\partial f}{\partial x_1}dx_1+\dots+\frac{\partial f}{\partial x_n}dx_n$ locally at almost every point. 

\begin{remark}
We deal with complexes which may be thought of as piecewise linear (PL) manifolds. In particular, such examples of complexes emerge as  
bi-Lipschitz triangulations of Riemannian manifolds with bounded geometry, see \cite{OA}.  
So, summing up, we should note that  the PL structure on a manifold defines  a Lipschitz structure  as well. 
Moreover, by contrast to an abstract PL manifold,  a metric simplicial complex allows us to modify it slightly, as it follows below. 
\end{remark}

\begin{definition}
By a Lipschitz structure $L_K$ on a metric simplicial complex $K$, we will mean a maximal by inclusion $\mathcal G$-atlas, which is formed by  bi-Lipschitz homeomorphisms $B \to K$, 
where $B \subset \mathbb R^n$ is a bounded set, 
and $\mathcal G$ is  the pseudogroup  of  bi-Lipschitz maps $\mathbb R^n \to \mathbb R^n$. 
That atlas is not empty since it includes an atlas formed by charts defined as stars of our complex at each point, which are bi-Lipschitz deformations of a unit ball in $\mathbb R^n$. 
\end{definition}

\begin{definition} Let $\mathscr L(K)$ be a ring of Lipschitz functions over a metric simplicial complex $K$. 
\begin{itemize}
\item Consider a free $\mathscr L$-module  $\Omega_{\mathscr L}$ generated by the set $ \{df\mid f \in \mathscr L(K)\}$.
Then there exists an exterior algebra 
$$
\bigwedge\nolimits_{\mathscr L}\Big( \Omega_{\mathscr L}\Big)
$$
on the module $\Omega_{\mathscr L}$.
\item Assume that $L_K = \{\varphi_i\colon U_i \to K\mid U_i \subset \mathbb R^n\}$ is a Lipschitz structure on K. Then for every element 
$\omega \in \bigwedge\nolimits_{\mathscr L}\Big( \Omega_{\mathscr L}\Big)$ such that $\omega = f_0df_1\wedge\dots \wedge df_k$,  we have:
	\begin{itemize}
	\item[1.] Let $\varphi \in L_K$, we will use the following notation  $\omega_{\varphi_i} = (\varphi_i^*f_0)d(\varphi_i^*f_1)\wedge\dots \wedge d(\varphi_i^*f_k)$;
	\item[2.] Consider the $\mathscr C^\infty$-module $\Omega^{n-k}_c(U_i)$ of compactly supported differential forms on $U_i$ and define
	a map $\Omega^{n-k}_c(U_i) \to \mathbb R$ by the rule
	$$
	\alpha \mapsto\langle\omega_{\varphi_i},\,\alpha\rangle = \int\limits_{U_i}\omega_{\varphi_i}\wedge\alpha;
	$$
	\end{itemize}
\item Let us introduce an equivalence relation $\sim_{L_K}$  
on 
$
\bigwedge\nolimits_{\mathscr L}\Big( \Omega_{\mathscr L}\Big)
$
 in the following manner. We will say that $\omega \sim_{L_K}\theta$  if, for  every $\varphi_i \in L_K$,
it holds on each $ \alpha \in \Omega^{n-k}_c(U_i)$, that $\langle\omega_{\varphi_i},\,\alpha\rangle = \langle\theta_{\varphi_i},\,\alpha\rangle$. 
\end{itemize}
Finally, we can define the  $\mathscr L$-module of Lipschitz differential forms as the quotient module
$$
\faktor{\bigwedge\nolimits_{\mathscr L}\Big( \Omega_{\mathscr L}\Big)}{\sim_{L_K}}.
$$
\end{definition}
\begin{remark}
The previous definition implies that 
$$
\faktor{\bigwedge\nolimits_{\mathscr L}\Big( \Omega_{\mathscr L}\Big)}{\sim_{L_K}}.
$$
has a finite dimension. In other words, for $k>n$ each homogeneous component  $f_0df_1\wedge\dots \wedge df_k$  lies in the equivalence class of zero. 
\end{remark}
\begin{remark}
Below, we will identify 
$$
\faktor{\bigwedge\nolimits_{\mathscr L}\Big( \Omega_{\mathscr L}\Big)}{\sim_{L_K}}.
$$
and 
$$
\bigwedge\nolimits_{\mathscr L}\Big( \Omega_{\mathscr L}\Big).
$$
\end{remark}
Then define the de Rham differential $d$ in the sense of distributions. Let $\omega \in \Omega^k_p(K)$ we will call a $(k+1)$-form $\eta$  the differential $d\omega$ if
$$
\int_U \omega \wedge d\phi = (-1)^{k+1} \int_U  \eta \wedge \phi
$$
for every compactly supported  Lipschitz $(n-k)$-form $\phi$. 
\begin{definition}
We will use $\Omega^k_\pi(K)$ for
 the closure of the space of Lipschitz differential forms on  $|K|$ with respect to  the following norm
$$\|\omega\|_{\Omega^k_{\pi}} = \|\omega\|_{\Omega^k_{p_k}}+\|d\omega\|_{\Omega^{k+1}_{p_{k+1}}}.$$
\end{definition}

\begin{definition}
Let $S{\mathscr L}^k_{\pi}(K)$ be  the closure of the space of smooth on any simplex  Lipschitz forms with respect to  the following norm
$$
\|\omega\|_{S{\mathscr L}^k_{\pi}(K)} = \Bigg\{\sum\limits_{T\in K([n])}  \|\omega\|_{\Omega_{\infty}(T)}^{p_k}\Bigg\}^{\frac{1}{p_k}}
+\Bigg\{\sum\limits_{T\in K([n])}  \|\omega\|_{\Omega_{\infty}(T)}^{p_{k+1}}\Bigg\}^{\frac{1}{p_k+1}}
$$
\end{definition}
\begin{definition}
Suppose that 
$
C(K) = \Big(C^i(K),\,\partial^{i}\Big)
$
is a  complex of simplicial cochains.
We will use the following notation for $L_p$-norms of  simplicial cochains.
$$c\in C(K),~\|c\|_{C^k_p(K)} = \Bigg\{\sum_{T\in K([n])} |c(T)|^p dx\Bigg\}^{\frac{1}{p}}$$
\end{definition}
\begin{definition}
For a simplicial cochain $c\in C(K)$,  we will assume, that the Sobolev norm is defined in the following manner 
$$
\|c\|_{C^k_{\pi}(K)} = \|c\|_{C^k_{p_k}(K)} +\|\partial c\|_{C^{k}_{p_{k+1}}(K)}
$$
and  $C_{\pi}(K)$  is a Banach space with respect to  that norm. 
\end{definition}

{\centering
\section{Regularization}\label{Reg}
}

\subsection{Definition of $\mathscr R$ and $\mathscr A$}\label{DefReg}
There is a diffeomorphism $h \in \Diff({\mathbb R}^n,\, {\bf B}_1)$, where    ${\bf B}_1$ is an open ball with centre $0$ and radius $1$. Let $U\subset {\mathbb R}^n$  and ${\bf B}_1 \subset U$.
Let $s_v$ be an action on $\mathbb R^n$ by
$$
s_v(x) = x+v.
$$
We can define a  family of diffeomorphisms $\mathfrak s_{ v}$  
$$
\mathbb R^n \to \Diff(U)
$$
such that 
$\mathfrak s_v\colon U \to U$
acts in the following manner
$$
\mathfrak s_v x =\begin{cases}
hs_vh^{-1}(x),~ \text{if $x \in {\bf B}_1$};\\
x,~ \text{if $x \notin {\bf B}_1$}
\end{cases}
$$
and, moreover, 
as shown in the de Rham's book \cite{dR},  $\mathfrak s^*_{tv}$ produces a group action of the additive group of real numbers on $U$, that is, 
$\mathfrak s^*_{(t_0+t_1)v} = \mathfrak s^*_{t_1v}\circ \mathfrak s^*_{t_0v}$.
And also we can say that  ${{\frak X}_v = d_{h^{-1}(x)}h(v)}$ is a vector field  consisting of tangent vectors to $\mathfrak s_{tv} (x)$.

Let $f\colon {\mathbb R}^n \to {\mathbb R}$ be a compactly supported smooth function such that \linebreak
${ {\rm supp}(f) \subset B_1}$, 
${\int_{{\bf R}^n}}f(v)dv^0 \dots dv^{n-1} = 1$, $f(v)\ge 0$ and ${f(v) = f(-v)}$. Let us put $\tau(v) = f(v)dv^0 \dots dv^{n-1}$.
So we can define operators
$$
{R} \omega= \int\limits_{{\mathbb R}^n}\mathfrak s^{*}_{\varepsilon v} \omega(x) \cdot \tau(v).
$$
As shown in \cite{GT}:
\begin{lemma}\label{TGIneq}
Suppose that ${\bf B}_1 \subset \mathbb R^n$ is an open ball and $\Omega^{*}_{ \text{\fontsize{5}{6}\selectfont {\rm dRh}}}({\bf B}_1) $ is the de Rham complex of $C^\infty$-forms. 
For   $\frac{1}{p}-\frac{1}{q}<\frac{1}{n}$ and  $\frac{1}{r}-\frac{1}{p}<\frac{1}{n}$, there
exists an operator $S\colon \Omega^{i}_{ \text{\fontsize{5}{6}\selectfont {\rm dRh}}}({\bf B}_1) \to \Omega^{i-1}_{ \text{\fontsize{5}{6}\selectfont {\rm dRh}}}({\bf B}_1) $
such that 
$$
\omega = Sd\omega+dS\omega 
$$
and, in case $\omega \in \Omega^i_{p,\,r}({\bf B}_1) \cap  \Omega^{i}_{ \text{\fontsize{5}{6}\selectfont {\rm dRh}}}({\bf B}_1)$, we  obtain 
$$
\|S\omega\|_{\Omega^{i-1}_{q,\,p}({\bf B}_1)} \le C\| \omega\|_{\Omega^i_{p,\,r}({\bf B}_1)}.
$$
\end{lemma}
Let us put
$$
A = (R - {\bf 1})S,
$$
and, as a result, we have 
\begin{align*}
dA+Ad = d(R - {\bf 1})S &+ (R - {\bf 1})Sd = dRS - dS + RSd-Sd \\
= RdS + RSd-(dS+Sd) &= R(dS+Sd) -(dS+Sd)  = R -{\bf 1}. 
\end{align*}

It was shown \cite{GKS1} that 
\begin{lemma}\label{GKSEst}
For every $\varepsilon > 0$  the maps $\mathcal{R}_{\varepsilon}$ and  $\mathcal{A}_{\varepsilon}$ are bounded on $\Omega_p^k({\bf B}_1)$ with respect to the $L_p$-norm
 and moreover the following estimations hold
$
\|\mathcal{R}_{\varepsilon}\|_p\le C(\varepsilon),
$
and
$
\|\mathcal{A}_{\varepsilon}\|_p\le M(\varepsilon);
$
where $C(\varepsilon) \to 1$, $M(\varepsilon) \to 0$ as $\varepsilon \to 0$.
\end{lemma}
\begin{remark}
Combining \emph{[{\bf Lemma \ref{TGIneq}}]} and   \emph{[{\bf Lemma \ref{GKSEst}}]} we obtain estimations on norms of  operators 
$R$ and $A$ acting on areas which are bi-Lipschitz homeomorphic to ${\bf B}_1$.
\end{remark}

Consider bi-Lipschitz homeomorphism 
$
\varphi \colon {\bf B}_1 \to U \subset {\mathbb R}^n
$
then we can define operators $\tilde{\mathcal R_{\varepsilon}}$ and $\tilde{\mathcal A_{\varepsilon}}$:

\begin{equation*}
\begin{split}
\xymatrix@R = 4em{
\Omega^k_{\pi}(U)\ar[r]^-{\tilde{\mathcal R_{\varepsilon}}}\ar[d]_{\varphi^*}&\bigwedge^k \Omega_{\mathscr L}(U)\\
\Omega^k_{\pi}({\bf B}_1)\ar[r]_-{\mathcal R_{\varepsilon}}&\Omega^k_{ \text{\fontsize{5}{6}\selectfont {\rm dRh}}} ({\bf B}_1)\ar[u]_{(\varphi^{-1})^*}
}
\end{split}
~~~~~~~~~~~~~~
\begin{split}
\xymatrix@R = 4em{
\Omega^k_{\pi}(U)\ar[r]^-{\tilde{\mathcal A_{\varepsilon}}}\ar[d]_{\varphi^*}& \Omega^{k-1}_{\pi}(U)\\
\Omega^k_{\pi}({\bf B}_1)\ar[r]_-{\mathcal A_{\varepsilon}}&\Omega^{k-1}_{\pi}({\bf B}_1)\ar[u]_{(\varphi^{-1})^*}
}
\end{split}
\end{equation*}
where $\Omega_{\mathscr L}(U)$ is a free module generated by differentials of Lipschitz functions.

It is not hard to see that the commutative squares are squares  in the category of normed Banach spaces because all arrows are bounded maps, and moreover, we have
$
\|\tilde{\mathcal{R}_{\varepsilon}}\|\le \tilde{C}(\varepsilon),
$
and
$
\|\tilde{\mathcal{A}_{\varepsilon}}\|\le \tilde{M}(\varepsilon).
$

Let $K$ be a star-bounded complex. Assume that $K'$ is the first barycentric subdivision of $K$. Let  $\Sigma'_i$ be the star of vertex $e_i$ in $K'$.  
Let $\varphi_i$ be a bi-Lipschitz homeomorphism $\varphi_i\colon {\rm Int}\,\Sigma_i \to U \subset \mathbb R^n$ 
such that ${\bf B_1}\subset U$ and   $\Sigma^{\prime}_i \subset  {\rm Int}\,\varphi^{-1}({\bf B}_1) $.

Given $\varepsilon > 0$, define operators $\mathcal R_i$ and $\mathcal A_i$
$$
\mathcal R_i \omega = \begin{cases}
 (\varphi_i^{-1})^* {\mathcal R_{\varepsilon}} \varphi_i^* \omega~\text{on}~\Sigma_i\\
\omega,\text{~otherwise}
\end{cases};~\mathcal A_i \omega = \begin{cases}
(\varphi_i^{-1})^* {\mathcal A_{\varepsilon}} \varphi_i^*\omega ~\text{on}~\Sigma_i\\
0, \text{~otherwise}
\end{cases}
$$ 
Consider operators
$
\mathscr{R}\omega  = \lim\limits_{i\to\infty} \mathcal R_1\mathcal R_2\dots \mathcal  R_i \omega
$
and
$
\mathscr{A}\omega  = \sum \limits^{\infty}_{i=1} \mathcal R_1\mathcal R_2\dots \mathcal  R_{i-1} \mathcal A_i\omega.
$
Due to the forthcoming  [{\bf Lemma \ref{RCorr}}] and  [{\bf Lemma \ref{ACorr}}],  we can see that such operators are well-defined. 

\begin{lemma}\label{Hm}
Under the above assumptions,  $d\mathscr{A} + \mathscr{A}d = \mathscr{R} - {\bf 1}$.
That is a homotopy equivalence $\mathscr{R} \sim {\bf 1}$.
\begin{proof}
\begin{align*}
d\mathscr{A} + \mathscr{A}d =  \lim\limits_{i\to\infty} \Bigg[  \sum \limits^{k}_{i=1} \mathcal R_1\mathcal R_2\dots \mathcal  R_{i-1}( d\mathcal A_i + \mathcal A_id)\Bigg]\\
= \lim\limits_{i\to\infty} \Bigg[  \sum \limits^{k}_{i=1} \mathcal R_1\mathcal R_2\dots \mathcal  R_{i-1}( \mathcal R_i - {\bf 1})\Bigg] = \mathscr{R} - {\bf 1}
\end{align*}
\end{proof}
\end{lemma}

\subsection{Basic properties of $\mathscr R$}
\begin{theorem}\label{CTriang}
Let $K$ be a simplicial complex of bounded geometry with $L = 1$, $
\pi = \langle p_0,\,p_1,\dots,\,p_n\rangle\subset (1,\,\infty) ,~\frac{1}{p_{i+1}}-\frac{1}{p_i}\le\frac{1}{n}$. Then there exists the diagram
\begin{center}
$$
\xymatrix@R = 2em @C = 1.2em{ \dots\ar[r]^-{d}&\Omega^{k-1}_{\pi}({ K}) \ar[rr]^-{d}\ar[d]^-{\mathscr{R}}
&&\Omega^{k}_{\pi}({ K})\ar[rr]^-{d}\ar[lldd]_{\mathscr{A}}\ar[d]^-{\mathscr{R}}
&&\Omega^{k+1}_{\pi}({ K})\ar[lldd]_{\mathscr{A}}\ar[r]^-{d}\ar[d]^-{\mathscr{R}}&\dots\\
\dots &S\mathscr{L}^{k-1}_{\pi} (K)\ar@{_(->}[d]&&S\mathscr{L}^{k}_{\pi}(K)\ar@{_(->}[d]&&S\mathscr{L}^{k+1}_{\pi}(K)\ar@{_(->}[d]&\dots\\
\dots\ar[r]_-{d}&\Omega^{k-1}_{\pi}(K)\ar[rr]_-{d}&&\Omega^k_{\pi}(K)\ar[rr]_-{d}&& \Omega^{k+1}_{\pi}(K)\ar[r]_-{d}&\dots
}
$$
\end{center}
with commutative squares in the category of Banach spaces ${\sf Ban}$. Moreover, the following holds
$
\mathscr R  - {\bf 1}_{\Omega^*_{\pi}} = d\mathscr A + \mathscr A d.
$
\end{theorem}
\begin{remark}
This theorem is a direct corollary of all forthcoming propositions which specify the properties of each arrow in the diagram. Every claim assumes bounded geometry as the suppose.  
\end{remark}
\begin{definition}
Assume that $K$ and $K'$ are a couple of metric simplicial complexes, and $\mu \colon K' \to K$ is an isometric embedding. We will denote by 
 ${\rm res}_{K',\,K}$
 the induced map
 $$
 \bigwedge\Omega_{\mathscr L}(K) \to \bigwedge\Omega_{\mathscr L}(K')
 $$
 and its continuation to a corresponding Banach spaces obtained as the closure of the space of Lipschitz forms with respect to the chosen norm.
\end{definition}
Suppose that $S\mathscr{L}^{*} (K)$ is the de Rham complex consisting of  smooth on any simplex Lipschitz forms on $K$
\begin{lemma}\label{RCorr}
The arrow $\Omega^{*}_{\pi}({ K}) \xrightarrow{\mathscr{R}} \Omega^{*}_{\pi}({ K})$  is a morphism in the category ${\sf Ban}$,
namely, $\mathscr R$ is  a bounded operator. Moreover $\mathscr{R}\big(\Omega^{*}_{\pi}({ K})\big) \subset S\mathscr{L}^{k} (K)$.
\end{lemma}
\begin{proof}
Consider a star $\Sigma_i$ of $K$.   Assume that 
there is a ball $X_i$ such that  \linebreak
${ \Sigma^{\prime}_i \subset {\rm Int }\,X_i \subset {\rm Int}\, \varphi^{-1}({\bf B}_1) } $.
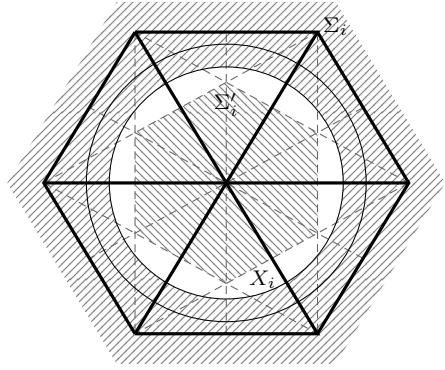
\begin{center}
\begin{tikzpicture}
\begin{scope}[scale = 0.8]
\fill[pattern=north west lines, pattern color =  gray] (0,5/3)--(1.5,2.5/3)--(1.5,-2.5/3)--(0,-5/3)--(-1.5,-2.5/3)--(-1.5,2.5/3)--(0,5/3);
\fill[pattern=north east lines, pattern color =  gray, even odd rule] (-3*1.2,0)--(-1.5*1.2, 2.5*1.2)--(1.5*1.2,2.5*1.2)--(3*1.2,0*1.2)--(1.5*1.2,-2.5*1.2)--(-1.5*1.2,-2.5*1.2)--(-3*1.2,0*1.2) (0,0) circle(5/2.6);
\draw[densely dashed,  color = black!60](0,-2.5)--(0,2.5);
\draw[densely dashed,  color = black!60](-3,0)--(1.5,2.5);
\draw[densely dashed,  color = black!60](3,0)--(-1.5,2.5);
\draw[densely dashed,  color = black!60](-1.5,2.5)--(-1.5,-2.5);
\draw[densely dashed,  color = black!60](1.5,2.5)--(1.5,-2.5);
\draw[densely dashed,  color = black!60](2.25,-1.25)--(-2.25,1.25);
\draw[densely dashed,  color = black!60](2.25,1.25)--(-2.25,-1.25);
\draw[densely dashed,  color = black!60](-3,0)--(1.5,-2.5);
\draw[densely dashed,  color = black!60](3,0)--(-1.5,-2.5);
\draw[very thick](-1.5,2.5)--(1.5,2.5)--(0,0)--(-1.5,2.5);
\draw[very thick](-1.5,-2.5)--(1.5,-2.5)--(0,0)--(-1.5,-2.5);
\draw[very thick] (-3,0)--(3,0);
\draw[very thick] (-1.5,2.5)--(-3,0)--(-1.5,-2.5);
\draw[very thick] (1.5,2.5)--(3,0)--(1.5,-2.5);
\draw (0,0) circle(5/2.6);
\draw (0,0) circle(2.3);
\draw (0, 1.3) node[font = \fontsize{8}{30}]  {$\Sigma_i'$};
\draw (1.8, 2.6 ) node[font = \fontsize{8}{30}]  {$\Sigma_i$};
\draw (0.6, -4.8/3 ) node[font = \fontsize{8}{30}]  {$X_i$};
\end{scope}
\end{tikzpicture}
\end{center}

Let $R_0,~R>0$ be such numbers that $R_0$ is a radius of $X_i$ and $R$ satisfies the following condition $ X_i \subset \{x\mid |x|\le R \} \subset {\rm Int}\, \varphi^{-1}({\bf B}_1) $, 
and $\alpha \colon K \to [0,\,1]$ 

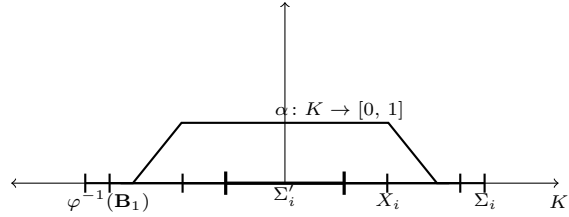
\begin{center}
\begin{tikzpicture}[scale = 0.8]
\draw[very thick,|-|] (-1,0)--(1,0);
\draw[thick,|-|] (-1.7,0)--(1.7,0);
\draw[thick,|-|] (-3.3,0)--(3.3,0);
\draw[thick,|-|] (-2.9,0)--(2.9,0);
\draw[->](0,0) -- (0,3);
\draw[thick] (-2.7,0)--(-2.5,0)--(-1.7,1)--(1.7,1)--(2.5,0)--(2.7,0);
\draw (0.9,1.2) node[font = \fontsize{7}{30}]{$\alpha\colon K \to [0,\,1]$};
\draw (0,-0.2) node[font = \fontsize{7}{30}]{$\Sigma^{\prime}_i$ };
\draw (1.7,-0.3) node[font = \fontsize{7}{30}]{$X_i$ };
\draw (-2.9,-0.3) node[font = \fontsize{7}{30}]{$ \varphi^{-1}({\bf B}_1)$ };
\draw (3.3,-0.3) node[font = \fontsize{7}{30}]{$\Sigma_i$ };
\draw (4.5,-0.3) node[font = \fontsize{7}{30}]{$K$ };
\draw[->](-1,0) -- (-4.5,0);
\draw[->](1,0) -- (4.5,0);

\end{tikzpicture}
\end{center}
be a function defined by
$$
\alpha(x) = \begin{cases}
1,~ x\in X_i;\\
\frac{R-|x|}{R-R_0},~x \in  \varphi^{-1}({\bf B}_1)~\text{and}~ R_0 \le |x|;\\
0,~\text{otherwise}.
\end{cases}
$$
We can represent  $\omega$ as a sum $\omega = \omega_1 + \omega_2$, 
where $\omega_1 = \alpha\omega$ and $\omega_2 = (1-\alpha)\omega$, i. e. ${{\rm supp} (\omega_2) \subset K \setminus X_i}$.

For any $\mathcal R_j$ and $\eta \in \Omega^k(K)$ such that  
${{\rm supp} (\eta) \subset K \setminus X_i}$, choosing $\varepsilon>0$ sufficiently small, 
we can achieve  ${{\rm supp} (\mathcal R_j \eta) \subset K \setminus \Sigma^{\prime}_i}$  
that implies
$R_j \eta  = 0$ on $\Sigma^{\prime}_i$. Due to this fact,  for each $j$, 
we can choose $\varepsilon_j$ in the definition of  the operator $\mathcal R_j$ in such a way that 
${{\rm supp} ( \mathcal R_{1}\dots  \mathcal R_{j} \omega_2) \subset K \setminus \Sigma^{\prime}_i}$  and correspondingly
${ \mathscr R \omega  = \mathscr  R\omega_1}$ on $\Sigma^{\prime}_i$. 

Let  $\Sigma_i$ be spanned by points $\{e_{j_1}, \dots,\, e_{j_n}\}$. 
For every form $\theta$ such that \linebreak
${\rm{supp}(\theta) \subset \varphi^{-1}({\bf B}_1) \subset \Sigma_i}$, 
we can see that only for $k \in \{{j_1}, \dots,\, {j_n}\}$ the operator $\mathcal R_k$ is distinct from the identity morphism.
Choosing sufficiently small $\varepsilon$ each time we face such an operator $R_k \in \{R_{j_1}\,\dots,\,R_{j_n}\}$ in the composition ${\mathcal R_{1} \mathcal  R_{2} \dots \mathcal R_j }$,
we can obtain a map  preserving the support of a form derived at  this step  inside ${\rm Int}\, \Sigma_i$.
Then, for each $j$, we have  \linebreak
${\mathcal R_{1}\dots\mathcal  R_{j} \omega_1 = \mathcal R_{j_1}\dots\mathcal  R_{j_n} \omega_1}$. 
 And so $\mathscr R \omega_1 = \mathcal R_{j_1}\dots\mathcal  R_{j_n} \omega_1$.

We know that, for  any $p$, it holds that 
$
\mathcal  R_{j_{k}}  \colon \Omega^{*}_p(\Sigma_i) \to  \Omega^{*}_p(\Sigma_i) 
$
and, moreover, following \cite{GKS1}, we can parametrize the operator $\mathcal  R_{j_{k},\,(\varepsilon)}$ by the $\varepsilon \to 0$ such that
$
\|\mathcal  R_{j_{k},\,(\varepsilon)}\|_{p} \le 1 + \varepsilon$.
Then  there exists $\varepsilon > 0$ such that 
$
 \| \mathscr R \|_p = \|\mathcal R_{j_1}\dots\mathcal  R_{j_n}\|_p \le 1 + O(\varepsilon),~ \varepsilon \to 0.
$
We should make a note that $d \mathscr R = \mathscr R d$ and the above argument holds for $d \omega$.
As a result, for each $i$, we have
$$
 \| {\rm res}_{\Sigma_i^{\prime},\,K} \circ \mathscr R \omega \|_{\Omega^k_{p_k}(\Sigma_i^{\prime})}  
 \le (1 + O(\varepsilon_i))\|{\rm res}_{\Sigma_i,\,K} \omega \|_{\Omega^k_{p_k}(\Sigma_i)}
 $$
 and
 \begin{align*}
 \| {\rm res}_{\Sigma_i^{\prime},\,K} \circ (d\mathscr R \omega) \|_{\Omega^{k+1}_{p_{k+1}}(\Sigma_i^{\prime})} 
 = \| {\rm res}_{\Sigma_i^{\prime},\,K} \circ \mathscr R d\omega \|_{\Omega^{k+1}_{p_{k+1}}(\Sigma_i^{\prime})}\\
 \le (1 + O(\varepsilon_i))\|{\rm res}_{\Sigma_i,\,K} d\omega \|_{\Omega^{k+1}_{p_{k+1}}(\Sigma_i)}
 \end{align*}
where $\varepsilon_i \to 0.$
The parameter $\varepsilon_i$ depends only on a star, and then, due to the star-boundedness of the complex, we can choose $\varepsilon = \min_i {\varepsilon_i}$. Assume that
$\omega \in \Omega^k_{p_k,\,p_{k+1}}(K)$
\begin{align*}
\| \mathscr R \omega\|_{\Omega^k_{\pi}(K)} &= 
\sum_{i}  \| {\rm res}_{\Sigma^{\prime}_i,\,K} \circ \mathscr R \omega \|_{\Omega^k_{p_k,\,p_{k+1}}(\Sigma^{\prime}_i)}\\
=\sum_{i} \Big( \| {\rm res}_{\Sigma^{\prime}_i,\,K} \circ \mathscr R \omega \|_{\Omega^k_{p_k}(\Sigma^{\prime}_i)} 
&+  \| {\rm res}_{\Sigma^{\prime}_i,\,K} \circ (d\mathscr  R \omega) \|_{\Omega^{k+1}_{p_{k+1}}(\Sigma^{\prime}_i)}\Big)\\
 \le \sum_{i}\Big((1 + O(\varepsilon_1))\|{\rm res}_{\Sigma_i,\,K} \omega \|_{\Omega^k_{p_k}(\Sigma_i)}
&+(1 + O(\varepsilon_2))\|{\rm res}_{\Sigma_i,\,K} d\omega \|_{\Omega^{k+1}_{p_{k+1}}(\Sigma_i)} \Big)\\
&\le \frac{(1 + O(\varepsilon))}{n}\|\omega \|_{\Omega^k_{\pi}(K)}
\end{align*}
In the light of what we have just said, $\mathscr R$ is a bounded map 
$
\mathscr R \colon \Omega^{*}_{\pi}(K) \to \Omega^{*}_{\pi}(K).
$
Moreover, it is not hard to see that 
$\mathscr R \colon \Omega^{*}_{\pi}(K) \to S\mathscr{L}^{*}(K)$. Indeed, 
we know that every $\mathcal R_i =  (\varphi_i^{-1})^* {\mathcal R_{\varepsilon}} \varphi_i^* $, then 
$$
 {\mathcal R_{\varepsilon}} \varphi_i^* \colon  \Omega^{*}_{\pi}(\Sigma_i) \to \Omega^{*}_{ \text{\fontsize{5}{6}\selectfont {\rm dRh}}} (U), 
$$
and $(\varphi_i^{-1})^*$ is a Lipschitz piecewise smooth map. 
\end{proof}
The proof of the following fact can be find in \cite{GP}.
\begin{lemma}\label{CompSubset}
Let ${\bf B}_1$ be an open ball in $\mathbb R^n$ with centre $0$ and radius $1$, \linebreak ${\bf B}_1 \subset U\subset {\mathbb R}^n$.
Then, for every $\varepsilon > 0$ and any compact $F \subset {\bf B}_1$,  the restriction map ${\rm res}_{F,\,U} \circ \mathcal{R}_{\varepsilon}$ is a bounded 
operator $\Omega^{k}_{p}(U) \to \Omega^k_{\infty}(F)$. 
\end{lemma}

\subsection{Decomposition of $\mathscr R$}\label{Decomp}
\begin{lemma}
The map $\mathscr R$ factors through $S\mathscr{L}^{k}_{\pi}(K)$  
\begin{center}
$$
\xymatrix@C = 0.1em{
& S\mathscr{L}^{k}_{\pi}(K)\ar@{->}[rd]^{\epsilon}&\\
\Omega^{k}_{\pi}({ K}) \ar@{-->}[ru]\ar[rr]^{\mathscr R}&& \Omega^{k}_{\pi}({ K})}
$$
\end{center}
\end{lemma}
\begin{proof}
As it was said above, the space ${S{\mathscr L}^k_{\pi}(K)}$ is endowed with the following norm
$$
\|\omega\|_{S{\mathscr L}^k_{\pi}(K)} = \bigg(\sum\limits_{\sigma \in K([n])}  \|\omega\|_{\Omega_{\infty}(\sigma)}^{p_k}\bigg)^{\frac{1}{p_k}}
+\bigg(\sum\limits_{\sigma\in K([n])}  \|\omega\|_{\Omega_{\infty}(\sigma)}^{p_{k+1}}\bigg)^{\frac{1}{p_k+1}}.
$$

Let a star  $\Sigma_i$ be spanned by points $\{e_{j_1}, \dots,\, e_{j_n}\}$. Every 
$n$-dimensional simplex $\sigma$ can be covered by stars of $K^{\prime}$: 
$\sigma \subset \bigcup_{l = 0}^{n-1} \Sigma^{\prime}_{j_l}$,
that, applying  [{\bf Lemma \ref{CompSubset}}], 
implies

 $$\| {\rm res}_{\sigma,\, K} \circ \mathscr R \omega \|^{p_k}_{\Omega^{k}_{\infty}(\sigma)} 
 \le \sum \limits_{l = 0}^{n} \| {\rm res}_{\Sigma^{\prime}_{j_l},\, K} \circ \mathscr R \omega \|^{p_k}_{\Omega^{k}_{\infty}(\Sigma^{\prime}_{j_l})}
  \le C_1 \sum \limits_{l = 0}^{n} \| {\rm res}_{\Sigma_{j_l},\, K}  \omega \|^{p_k}_{\Omega^{k}_{p_k}(\Sigma_{j_l})}$$
  and respectively
  \begin{align*}
  \| {\rm res}_{\sigma,\, K} \circ (d\mathscr R \omega) \|^{p_{k+1}}_{\Omega^{k+1}_{\infty}(\sigma)} &=   \| {\rm res}_{\sigma,\, K} \circ \mathscr R d\omega \|^{p_{k+1}}_{\Omega^{k+1}_{\infty}(\sigma)} \\
 \le \sum \limits_{l = 0}^{n} \| {\rm res}_{\Sigma^{\prime}_{j_l},\, K} \circ \mathscr R d\omega \|^{p_{k+1}}_{\Omega^{k+1}_{\infty}(\Sigma^{\prime}_{j_l})}
  &\le C_2 \sum \limits_{l = 0}^{n} \| {\rm res}_{\Sigma_{j_l},\, K} d \omega \|^{p_{k+1}}_{\Omega^{k+1}_{p_{k+1}}(\Sigma_{j_l})}.
\end{align*}
As a result we have
\begin{align*}
\sum\limits_{\sigma_i \in K([n])} \| {\rm res}_{\sigma_i,\, K} \circ \mathscr R \omega \|^{p_k}_{\Omega^{k}_{\infty}(\sigma_i)} 
  \le C_1 \sum\limits_{\sigma _i\in K([n])}\sum\limits_{l = 0}^{n} \| {\rm res}_{\Sigma^i_{j_l},\, K}  \omega \|^{p_k}_{\Omega^{k}_{p_k}(\Sigma^i_{j_l})} \\
  \le C_1  N \sum_j \| {\rm res}_{\Sigma_{j},\, K}  \omega \|^{p_k}_{\Omega^{k}_{p_k}(\Sigma_{j})}\\
    \le C_1  N n \sum_i \| {\rm res}_{\sigma_{i},\, K}  \omega \|^{p_k}_{\Omega^{k}_{p_k}(\sigma_{i})} = C_1^{\prime}\|\omega \|^{p_k}_{\Omega^{k}_{p_k}(K)}.
  \end{align*}
  and
  \begin{align*}
\sum_i \| {\rm res}_{\sigma_i,\, K} \circ (d\mathscr R \omega) \|^{p_{k+1}}_{\Omega^{k+1}_{\infty}(\sigma_i)} 
  \le C_2 \sum_i \sum \limits_{l = 0}^{n} \| {\rm res}_{\Sigma^i_{j_l},\, K}  d\omega \|^{p_{k+1}}_{\Omega^{k+1}_{p_{k+1}}(\Sigma^i_{j_l})} \\
  \le C_2  N \sum_j \| {\rm res}_{\Sigma_{j},\, K}  d\omega \|^{p_{k+1}}_{\Omega^{k+1}_{p_{k+1}}(\Sigma_{j})}\\
    \le C_2 N n \sum_i \| {\rm res}_{\sigma_{i},\, K}  d\omega \|^{p_{k+1}}_{\Omega^{k+1}_{p_{k+1}}(\sigma_{i})} = C_2^{\prime}\|d\omega \|^{p_{k+1}}_{\Omega^{k+1}_{p_{k+1}}(K)}.
  \end{align*}
  $$
\|\mathscr R \omega\|_{S{\mathscr L}^k_{\pi}(K)} \le C_1^{\prime}\|\omega \|_{\Omega^{k}_{p_k}(K)}+C_2^{\prime}\|d\omega \|_{\Omega^{k+1}_{p_{k+1}}(K)}
  $$
\end{proof}
\begin{lemma}
For every $m$-dimensional skeleton $K[m]$ of $K$, the operator \linebreak
${{\rm res}_{K[m],\, K}\circ \mathscr R}$
is a morphism $\Omega^{k}_{\pi}({ K}) \to \Omega^{k}_{\pi}(K[m])$ in ${\sf Ban}$ (a bounded operator).
\end{lemma}
\begin{proof}
In order to prove the lemma, we should complete the diagram with dashed arrows in such a way that the final diagram commutes 
\begin{center}
$$
\xymatrix@C = 0.7em{
& \Omega^{k}_{\pi}(K[m])&\\
& S\mathscr{L}^{k}_{\pi}(K)\ar[r]^{\delta}\ar@{->}[rd]^{\epsilon}\ar@{-->}[u]^{\gamma}&\im\epsilon\ar@/{}_{1.1pc}/@{->}[lu]_-{{\rm res}_{K[m],\, K}}\ar[d]^{\iota}\\
\Omega^{k}_{\pi}({ K}) \ar@{->}[ru]^{\beta}\ar[rr]^{\mathscr R}\ar@/{}^{2.1pc}/@{-->}[ruu]^{\alpha}&& \Omega^{k}_{\pi}({ K})
}
$$
\end{center}
The previous lemma provides us with the decomposition $\mathscr R = \epsilon\beta$. Also, we have $\epsilon = \iota\delta$.
In fact, the domain of  ${\rm res}_{K[m],\, K}$ is $\im \epsilon$. Then we can define $\alpha = \gamma\beta$. 
Below, we verify that it is valid.
\begin{align*}
\|{\rm res}_{K[m],\, K}&\circ \mathscr R \omega\|_ {\Omega^{k}_{\pi}(K[m])} \\
= \|{\rm res}_{K[m],\, K}\circ \mathscr R \omega\|_ {\Omega^{k}_{p_k}(K[m])}
&+\|{\rm res}_{K[m],\, K}\circ \mathscr (dR \omega)\|_ {\Omega^{k+1}_{p_{k+1}}(K[m])} 
\end{align*}
 Hence 
 \begin{align*}
 \|{\rm res}_{K[m],\, K}\circ \mathscr R \omega\|^{p_k}_ {\Omega^{k}_{p_k}(K[m])} = \sum_{i} \|{\rm res}_{\tau_i,\, K}\circ \mathscr R \omega\|^{p_k}_ {\Omega^{k}_p(\tau_i)}\\
 \le\sum_{i}  ({\rm mes}\, \tau_i) \|{\rm res}_{\tau_i,\, K}\circ \mathscr R \omega\|^{p_k}_ {\Omega^{k}_{\infty}(\tau_i)}.
 \end{align*}
 It is not hard to see that
 $$
 \|{\rm res}_{\tau_i,\, K}\circ \mathscr R \omega\|^{p_k}_ {\Omega^{k}_{\infty}(\tau_i)} 
 \le \sum_{\sigma \in K[n],\,\tau_i \hookrightarrow \sigma}\|{\rm res}_{\sigma,\, K}\circ \mathscr R \omega\|^{p_k}_ {\Omega^{k}_{\infty}(\sigma)}.
$$
As a result, we have
\begin{align*}
\|{\rm res}_{K[m],\, K}\circ \mathscr R \omega\|^{p_k}_ {\Omega^{k}_{p_k}(K[m])} 
&\le \bigg( \frac{\sqrt{m+1}}{m!\sqrt{2^{m}}}\bigg)
\sum_{i}  \sum_{\sigma \in K[n],\,\tau_i \hookrightarrow \sigma}\|{\rm res}_{\sigma,\, K}\circ \mathscr R \omega\|^{p_k}_ {\Omega^{k}_{\infty}(\sigma)} \\
&\le \bigg( \frac{\sqrt{m+1}}{m!\sqrt{2^{m}}}\bigg)\binom{n+1}{m+1}C\|\omega \|^{p_k}_{\Omega^{*}_{p_k}(K)}.
\end{align*}
Here we should again  make a note that $d \mathscr R = \mathscr R d$, and the above argument holds for $d \omega$.
\end{proof}
\begin{corollary}
For every $m$-dimensional skeleton $K[m]$ of $K$, the operator \linebreak
${{\rm res}_{K[m],\, K}\circ \mathscr R}$
is a morphism $\Omega^{k}_{\pi}({ K}) \to S\mathscr L^{k}_{\pi}(K[m])$ in ${\sf Ban}$ as well.
\end{corollary}
\begin{proof}
 The result follows directly from the below
 $$
 \|{\rm res}_{\tau_i,\, K}\circ \mathscr R \omega\|^{p_k}_ {\Omega^{k}_{\infty}(\tau_i)} 
 \le \sum_{\sigma \in K[n],\,\tau_i \hookrightarrow \sigma}\|{\rm res}_{\sigma,\, K}\circ \mathscr R \omega\|^{p_k}_ {\Omega^{k}_{\infty}(\sigma)}.
$$
\end{proof}
\subsection{Properties of the homotopy $\mathscr A$}
\begin{lemma}\label{ACorr}
The arrow 
$\Omega^{k}_{{p_k},\,{p_{k+1}}}({ K}) \xrightarrow{\mathscr{A}} \Omega^{k-1}_{{p_{k-1}},\,{p_{k}}}({ K})$ is a morphism in  the category ${\sf Ban}$. 
\end{lemma}

\begin{proof}
Let $\omega \in \Omega^k_{{p_k},\,{p_{k+1}}}(M)$ consider  $\mathcal A_i\omega$. It is not hard to see that  \linebreak
${\rm{supp}(\mathcal A_i \omega) \subset \varphi^{-1}({\bf B}_1)}$. 
Indeed, $\mathfrak s_{tv}^*$  acts on  the complement $K \setminus \varphi^{-1}({\bf B}_1)$  leaving points fixed. 
It follows that  $ \frak X_v = 0$ and,  consequently, $\iota_{\frak X_v}$ is a zero map at every point belonging to $K \setminus \varphi^{-1}({\bf B}_1)$.

Let  $\Sigma_i$ be spanned by points $\{e_{j_1}, \dots,\, e_{j_n}\}$. 
For every form 
compactly supported inside  ${\varphi^{-1}({\bf B}_1) \subset \Sigma_i}$ there are only a fixed number of operators, namely  $\mathcal R_{j_1},\dots,\, R_{j_n}$, which are distinct from the identity.
Similarly to the above, we can choose sufficiently small $\varepsilon$  for each operator $R_{j_k}$ in the composition ${\mathcal R_{1}\dots\mathcal  R_{i-1}\mathcal A_i }$,
that allows us to preserve the support of a form derived under every partial composition inside ${\rm Int}\, \Sigma_i$.
Then we have 
${\mathcal R_{1}\dots\mathcal  R_{i-1}\mathcal A_i \omega = \mathcal R_{j_1}\dots\mathcal  R_{j_n} \mathcal A_i \omega}$ and 
${\mathcal R_{j_1}\dots R_{j_n} \mathcal A_i \omega = 0}$  outside $\Sigma_i$. Now we can estimate the norm $\|\mathscr A \omega\|_{\Omega^{k-1}_{p_{k-1}}(K)}$ due to \cite{GT}:
$$
\|\mathcal R_{j_1}\dots R_{j_n} \mathcal A_i \omega\|_{\Omega^{k-1}_{p_{k-1}}(\Sigma_i)} \le C^{n}_{p_{k}} M_{p_{k}}\|\omega\|_{\Omega^{k}_{p_{k}}(\Sigma_i)}.
$$
So we have
\begin{align*}
&\|\mathscr A \omega\|_{\Omega^{k-1}_{{p_{k-1}},\,{p_k}}(K)} = \|\mathscr A \omega\|_{\Omega^{k-1}_{{p_{k-1}}}(K)} + \| d\mathscr A \omega\|_{\Omega^{k}_{{p_k}}(K)} \\
= &(\int\limits_{K}\left|\sum_{i} \mathcal R_{1}\dots R_{i-1} \mathcal A_i \omega  \right|^{p_{k-1}} dx)^{\frac{1}{{p_{k-1}}}} + (\int_{K}\left|\sum_{i} d\mathcal R_{1}\dots R_{i-1} \mathcal A_i \omega  \right|^{p_k} dx)^{\frac{1}{{p_k}}}\\
= &(\sum_{\sigma \in K[n]} \int\limits_{\sigma} \left|\sum_{i} {\rm res}_{\sigma,\, K} \mathcal R_{1}\dots R_{i-1} \mathcal A_i \omega  \right|^{p_{k-1}} dx)^{\frac{1}{{p_{k-1}}}}\\
+ &(\sum_{\sigma \in K[n]} \int\limits_{\sigma}\left|\sum_{i} {\rm res}_{\sigma,\, K} d\mathcal R_{1}\dots R_{i-1} \mathcal A_i \omega  \right|^{p_k} dx)^{\frac{1}{{p_k}}}.
\end{align*}
Every simplex $\sigma = \{e_0,\dots,\,e_n\}$ is the intersection of stars assigned to its vertices \linebreak
$\sigma = \bigcap_{e_i\in  \{e_0,\dots,\,e_n\}} \Sigma_i$, and moreover, 
$\sigma$ lies in no other star. And it follows that 
$$
\sum_{i} {\rm res}_{\sigma,\, K} \mathcal R_{1}\dots R_{i-1} \mathcal A_i \omega  = \sum_{e_j \in \{e_0,\dots,\,e_n\}} {\rm res}_{\sigma,\, K} \mathcal R_{1}\dots R_{j-1} \mathcal A_j \omega 
$$ 
and
$$
\sum_{i} {\rm res}_{\sigma,\, K} d \mathcal R_{1}\dots R_{i-1} \mathcal A_i \omega  = \sum_{e_j \in \{e_0,\dots,\,e_n\}} {\rm res}_{\sigma,\, K} d \mathcal R_{1}\dots R_{j-1} \mathcal A_j \omega.
$$ 
Hence
\begin{align*}
\|\mathscr A \omega\|_{\Omega^{k-1}_{{p_{k-1}},\,{p_k}}(K)} 
=& (\sum_{\sigma \in K[n]}  \int\limits_{\sigma} \left| \sum_{e_j \in \{e_0,\dots,\,e_n\}}{\rm res}_{\sigma,\, K} \mathcal R_{1}\dots R_{j-1} \mathcal A_j \omega  \right|^{p_{k-1}} dx)^{\frac{1}{{p_{k-1}}}}\\ 
+&  (\sum_{\sigma \in K[n]} \int\limits_{\sigma}\left| \sum_{e_j \in \{e_0,\dots,\,e_n\}} {\rm res}_{\sigma,\, K} d\mathcal R_{1}\dots R_{j-1} \mathcal A_j \omega  \right|^{p_k} dx)^{\frac{1}{{p_k}}}\\
\le (n+1)^{\frac{{p_{k-1}}-1}{{p_{k-1}}}} &(\sum_{\sigma \in K[n]}  \int\limits_{\sigma}  \sum_{e_j \in \{e_0,\dots,\,e_n\}}|{\rm res}_{\sigma,\, K} \mathcal R_{1}\dots R_{j-1} \mathcal A_j \omega  |^{p_{k-1}} dx)^{\frac{1}{{p_{k-1}}}}\\
+ (n+1)^{\frac{{p_k}-1}{{p_k}}} &(\sum_{\sigma \in K[n]}  \int\limits_{\sigma} \sum_{e_j \in \{e_0,\dots,\,e_n\}} |{\rm res}_{\sigma,\, K} d\mathcal R_{1}\dots R_{j-1} \mathcal A_j \omega  |^{p_k} dx)^{\frac{1}{{p_k}}}.
\end{align*}
Let us check the following
\begin{align*}
  \int\limits_{\sigma}  &\sum\limits_{e_j \in \{e_0,\dots,\,e_n\}}|{\rm res}_{\sigma,\, K} \mathcal R_{1}\dots R_{j-1} \mathcal A_j \omega  |^{p_{k-1}} dx\\
  = &\sum_{e_j \in \{e_0,\dots,\,e_n\}}   \int\limits_{\sigma}|{\rm res}_{\sigma,\, K} \mathcal R_{1}\dots R_{j-1} \mathcal A_j \omega  |^{p_{k-1}} dx\\
  \le &\sum_{e_j \in \{e_0,\dots,\,e_n\}}   \int\limits_{\Sigma_j}| \mathcal R_{1}\dots R_{j-1} \mathcal A_j \omega  |^{p_{k-1}} dx\\
  = &\sum_{e_j \in \{e_0,\dots,\,e_n\}}   \int\limits_{\Sigma_j}| \mathcal R_{j_0}\dots R_{j_n} \mathcal A_j \omega  |^{p_{k-1}} dx\\
  = &\sum_{e_j \in \{e_0,\dots,\,e_n\}}  \|\mathcal R_{j_0}\dots R_{j_n} \mathcal A_j \omega  \|^{p_{k-1}}_{\Omega^{k-1}_{p_{k-1}}(\Sigma_j)}\\
  \le &\sum_{e_j \in \{e_0,\dots,\,e_n\}}  C^{n}_{p_k} M_{p_k}\|\omega\|^{p_{k-1}}_{\Omega^{k}_{{p_k}}(\Sigma_j)}.
\end{align*}
That implies
\begin{align*}
\Bigg( \sum_{\sigma \in K[n]}  \int\limits_{\sigma}  &\sum_{e_j \in \{e_0,\dots,\,e_n\}}|{\rm res}_{\sigma,\, K} \mathcal R_{1}\dots R_{j-1} \mathcal A_j \omega  |^{p_{k-1}} dx\Bigg)^{\frac{1}{{p_{k-1}}}}\\
\le  \Bigg(\sum_{\sigma \in K[n]}  &\sum_{e_j \in \{e_0,\dots,\,e_n\}}  C^{\prime}\|\omega\|^{p_{k-1}}_{\Omega^{k}_{p_{k}}(\Sigma_j)}\Bigg)^{\frac{1}{{p_{k-1}}}}
\le {C^{\prime}}^{\frac{1}{{p_{k-1}}}} (n+1)\|\omega\|_{\Omega^{k}_{p_{k}}(K)}.
\end{align*}
 
Similarly, we can estimate $\| d\mathscr A \omega\|_{\Omega^{k-1}_{{p_k}}(K)}$.
And, as  a result, we have
$$
\|\mathscr A \omega\|_{\Omega^{k-1}_{{p_{k-1}},\,{p_k}}(K)} \le C^{\prime\prime}\| \omega\|_{\Omega^{k}_{{p_{k}},\,{p_{k+1}}}(K)}
$$
\end{proof}
{\centering
\section{The de Rham Operator $\mathscr I$ }\label{dRhOp}
}

Let us define an operator
$$
 \mathscr I \colon S\mathscr{L}^k_{\pi}(K) \to C^k_{\pi}(K)
$$
by the following
$$
\omega \mapsto \{\sigma \mapsto \int_\sigma \omega \}. 
$$
\begin{lemma}\label{Split}
Let $K$ be a simplicial complex of bounded geometry, then the map $\mathscr I$ is a split epimorphism,
in other words, $S\mathscr{L}^k_{\pi}(K) \cong C^k_{\pi}(K) \oplus {\rm ker}^k( \mathscr I)$
\end{lemma}
\begin{proof}
First of all, we should confirm that $\mathscr I$ is in fact a morphism between Sobolev spaces. Indeed, 

$$
\|\mathscr I\omega\|_{C^k_{{p_k},\,{p_{k+1}}}(K)} = \|\mathscr I\omega\|_{C^k_{p_k}(K)} +\|d \mathscr I\omega\|_{C^k_{p_{k+1}}(K)}
$$

\begin{align*}
\|\mathscr I\omega\|_{C^k_{p_k}(K)}^{p_k} = &\sum\limits_{\tau \in K[k]} \left| \int\limits_{\tau}\omega\right|^{p_k} \le
 \sum\limits_{\tau \in K[k]}  \left(\int\limits_{\tau}|\omega|\right)^{p_k} \\
\le\{{\rm mes}\,\tau\}^{{p_k}-1}  \bigg( &\sum\limits_{\tau\in K[k]} \int\limits_{\tau}|\omega|^{p_k}\bigg)
\le \left( \frac{\sqrt{k+1}}{k!\sqrt{2^{k}}} \right)^{{p_k}-1} \|\omega\|^{p_k}_{S\mathscr{L}^k_{p_k}(K)}.
\end{align*}
\begin{align*}
\|d \mathscr I\omega\|_{C^k_{p_{k+1}}(K)}^{p_{k+1}} =\| \mathscr I d\omega\|_{C^k_{p_{k+1}}(K)}^{p_{k+1}} =  &\sum\limits_{\tau \in K[k]} \left| \int\limits_{\tau}d\omega\right|^{p_{k+1}} \le
 \sum\limits_{\tau \in K[k]}  \left(\int\limits_{\tau}|d\omega|\right)^{p_{k+1}} \\
\le\{{\rm mes}\,\tau\}^{{p_{k+1}}-1}  \bigg( &\sum\limits_{\tau\in K[k]} \int\limits_{\tau}|d\omega|^{p_{k+1}}\bigg)
\le \left( \frac{\sqrt{k+1}}{k!\sqrt{2^{k}}} \right)^{{p_{k+1}}-1} \|d\omega\|^{p_{k+1}}_{S\mathscr{L}^{k+1}_{p_{k+1}}(K)}.
\end{align*}

Assume $c \in C^k_{p_{k}}(K)$, let us represent $c$ as a sum in the Banach space which converges absolutely  
$
c = \sum_{\sigma \in K([k])} c(\sigma)\chi_{\sigma}$, where
$$
\chi_{\sigma}(\sigma') = \begin{cases} 1,~\text{if}~\sigma' = \sigma;\\0,~\text{otherwise},\end{cases}
$$ 
that is, $\sum_{\sigma \in K([k])} |c(\sigma)|^{p_{k}} < \infty$.  
In other words, we can define $C^k_{p_{k}}(K)$ as a closure of $\mathbb R$-vector space $C_c^k(K)$ generated by elements $\{ \chi_\sigma\}_{\sigma \in K([k])}$, 
or, equivalently,   a closure of $\mathbb R$-vector space of compactly supported cochains. 

Here we turn to  Hassler Whitney's idea (see \cite{HW}), in order to define an $\mathbb R$-linear map ${\mathscr W}$ as the following. 
Firstly, let us write it down for elements of basis
$$
{\mathscr W}(\chi_{\sigma}) = k!\sum_{i=0}^k (-1)^i t_i dt_0\wedge\dots\wedge dt_{i-1}\wedge dt_{i+1}\wedge \dots \wedge dt_{k},
$$
where $\{t_i\}$ is a set of barycentric coordinates assigned to the vertices of the simplex $\sigma$.    
The ${\rm supp}({\mathscr W}(\chi_{\sigma}))$ contains only one simplex  belonging to the $k$-skeleton of K, namely $\sigma$. 
It implies that any sum ${\mathscr W}(\chi_{\sigma})+{\mathscr W}(\chi_{\sigma'})$ can be uniquely restricted to every $k$-simplex.
Then we can extend our map to the space $C_c^k(K)$ by linearity. It is clear that  ${\mathscr W} (C_c^k(K))$ is a set of  compactly supported differential forms, and moreover, the support of 
every form ${\mathscr W}(\chi_{\sigma})$
contains at most $N$ simplexes due to star-boundedness of the complex.  
Let us check that ${\mathscr W}\colon C^k_c(K) \to S\mathscr{L}^k_{{p_k},\,{p_{k+1}}}(K)$ is a bounded map.  

\begin{align*}
&\|{\mathscr W}(c)\|_{S\mathscr L_{{p_k},\,{p_{k+1}}}(K)} = \Bigg(\sum_{j = 1}^{r}\max_{\sigma_j} |{\mathscr W}(c)|^{p_k}\Bigg)^{\frac{1}{{p_k}}}+  \Bigg(\sum_{j = 1}^{r}\max_{\sigma_j} |d{\mathscr W}(c)|^p\Bigg)^{\frac{1}{{p_{k+1}}}}\end{align*}

Assume that 
$
c = \sum_{i = 1}^{m} c(\tau_i) \chi_{\tau_i}
$
then 
$$
{\mathscr W} c = {\mathscr W} \Bigg(\sum_{i = 1}^{m} c(\tau_i) \chi_{\tau_i}\Bigg) =  \sum_{i = 1}^{m} c(\tau_i)  {\mathscr W} (\chi_{\tau_i})
$$

As mentioned above, there is only a finite number $r$ of simplexes belonging to the support $\bigcup_{i  = 1}^m {\rm supp} ({\mathscr W}(\chi_{\tau_i}))$ of  ${\mathscr W}(c)$.
Applying Jensen's inequality, we have
\begin{align*}
\max_{\sigma} |{\mathscr W}(c)|^{p_k} &=  \max_{\sigma} \Bigg|  \sum_{i = 1}^{m} c(\tau_i)  {\mathscr W} (\chi_{\tau_i})\Bigg|^{p_k}\\
&\le \max_{\sigma} \Bigg\{  \binom{n+1}{k+1}^{{p_k}-1}  \sum_{i = 1}^{m} |c(\tau_i)|^{p_k} |{\mathscr W} (\chi_{\tau_i})|^{p_k} \Bigg\}.
\end{align*}
\begin{align*}
 \Bigg(\sum_{j = 1}^{r}\max_{\sigma_j} |{\mathscr W}(c)|^{p_k}\Bigg)^{\frac{1}{{p_k}}}
& \le \Bigg(\sum_{j = 1}^{r}\max_{\sigma_j} \Big\{  \binom{n+1}{k+1}^{{p_k}-1}  \sum_{i = 1}^{m} |c(\tau_i)|^{p_k} |{\mathscr W} (\chi_{\tau_i})|^{p_k} \Big\}\Bigg)^{\frac{1}{{p_k}}}\\
& \le \Bigg(\sum_{j = 1}^{r}   \binom{n+1}{k+1}^{{p_k}-1}  \sum_{i = 1}^{m} |c(\tau_i)|^{p_k} \max_{\sigma_j}|{\mathscr W} (\chi_{\tau_i})|^{p_k} \Bigg)^{\frac{1}{{p_k}}}\\
& \le \binom{n+1}{k+1}^{\frac{{p_k}-1}{{p_k}}} \Bigg( \sum_{i = 1}^{m}|c(\tau_i)|^{p_k} \sum_{j = 1}^{r} \max_{\sigma_j}|{\mathscr W} (\chi_{\tau_i})|^{p_k}    \Bigg)^{\frac{1}{{p_k}}}
\end{align*}

It is not hard to see that there is a constant $C$ such that
$
\max_{\sigma_j}|{\mathscr W}(\chi_{\tau_i})|^{p_k} \le C$,~ {if $\tau_i$ is a $k$-face of $\sigma_j$},
and
$
|{\mathscr W}(\chi_{\tau_i})|^{p_k} = 0$, {otherwise}.
As a result, we have
\begin{align*}
 \Bigg(\sum_{j = 1}^{r}\max_{\sigma_j} |{\mathscr W}(c)|^{p_k}\Bigg)^{\frac{1}{{p_k}}}
\le  \binom{n+1}{k+1}^{\frac{{p_k}-1}{{p_k}}}  C^{\frac{1}{{p_k}}} N \Bigg( \sum_{i = 1}^{m}|c(\tau_i)|^{p_k} \Bigg)^{\frac{1}{{p_k}}}
\end{align*}

\begin{align*}
 \Bigg(\sum_{j = 1}^{r}\max_{\sigma_j} |d{\mathscr W}(c)|^{p_{k+1}}\Bigg)^{\frac{1}{{p_{k+1}}}} 
 =   \Bigg(\sum_{j = 1}^{r}\max_{\sigma_j} \Bigg|  d\sum_{i = 1}^{m} c(\tau_i)  {\mathscr W} (\chi_{\tau_i})\Bigg|^{p_{k+1}}\Bigg)^{\frac{1}{{p_{k+1}}}}\\
  =   \Bigg(\sum_{j = 1}^{r}\max_{\sigma_j} \Bigg|  \sum_{i = 1}^{m} c(\tau_i)  d{\mathscr W} (\chi_{\tau_i})\Bigg|^{p_{k+1}}\Bigg)^{\frac{1}{{p_{k+1}}}}
\end{align*}

$$
\|dc\|_{C^k_c(K)} = \| \sum_{i = 1}^{m} c(\tau_i) d\chi_{\tau_i} \|_{C^k_c(K)} =\Bigg( N\sum_{i = 1}^{m}|c(\tau_i)|^{p_{k+1}} \Bigg)^{\frac{1}{{p_{k+1}}}}
$$

\begin{align*}
 \Bigg(\sum_{j = 1}^{r}\max_{\sigma_j} |d{\mathscr W}(c)|^{p_{k+1}}\Bigg)^{\frac{1}{{p_{k+1}}}} 
  =   \Bigg(\sum_{j = 1}^{r}\max_{\sigma_j} \Bigg|  \sum_{i = 1}^{m} c(\tau_i)  k\vol \tau_i \Bigg|^{p_{k+1}}\Bigg)^{\frac{1}{{p_{k+1}}}}\\
  \le  \Bigg(\sum_{j = 1}^{r} \max_{\sigma_j} \Big\{  \binom{n+1}{k+1}^{{p_{k+1}}-1}  \sum_{i = 1}^{m} |c(\tau_i)|^{p_{k+1}} |k\vol \tau_i|^{p_{k+1}} \Big\}\Bigg)^{\frac{1}{{p_{k+1}}}}\\
    \le  \Bigg(\sum_{j = 1}^{r}   \binom{n+1}{k+1}^{{p_{k+1}}-1}  \sum_{i = 1}^{m} |c(\tau_i)|^{p_{k+1}}  \max_{\sigma_j}|k\vol \tau_i|^{p_{k+1}} \Bigg)^{\frac{1}{{p_{k+1}}}}\\
     \le  C_1 \binom{n+1}{k+1}^{\frac{{p_{k+1}}-1}{{p_{k+1}}}} \Bigg(\sum_{j = 1}^{r}   \sum_{i = 1}^{m} |c(\tau_i)|^{p_{k+1}}   \Bigg)^{\frac{1}{{p_{k+1}}}}
  \end{align*}
\begin{align*}
 \Bigg(\sum_{j = 1}^{r}\max_{\sigma_j} |d{\mathscr W}(c)|^{p_{k+1}}\Bigg)^{\frac{1}{{p_{k+1}}}} 
  \le   C_1 \binom{n+1}{k+1}^{\frac{{p_{k+1}}-1}{{p_{k+1}}}} \Bigg( N \sum_{i = 1}^{m} |c(\tau_i)|^{p_{k+1}}   \Bigg)^{\frac{1}{{p_{k+1}}}}
 \end{align*}
As a result we have
\begin{align*}
 \Bigg(\sum_{j = 1}^{r}\max_{\sigma_j} |{\mathscr W}(c)|^{p_k}\Bigg)^{\frac{1}{{p_k}}}
\le C^{\frac{1}{{p_k}}} \binom{n+1}{k+1}^{\frac{{p_k}-1}{{p_k}}}   N \|c\|_{C^k_c(K)} 
\end{align*}
and
\begin{align*}
 \Bigg(\sum_{j = 1}^{r}\max_{\sigma_j} |d{\mathscr W}(c)|^{p_{k+1}}\Bigg)^{\frac{1}{{p_{k+1}}}} 
  \le   C_1 \binom{n+1}{k+1}^{\frac{{p_{k+1}}-1}{{p_{k+1}}}} \|dc\|_{C^{k+1}_c(K)} 
 \end{align*}

\begin{align*}
&\|{\mathscr W}(c)\|_{S\mathscr L_{{p_k},\,{p_{k+1}}}(K)} = \Bigg(\sum_{j = 1}^{r}\max_{\sigma_j} |{\mathscr W}(c)|^{p_k}\Bigg)^{\frac{1}{{p_k}}}+  \Bigg(\sum_{j = 1}^{r}\max_{\sigma_j} |d{\mathscr W}(c)|^{p_{k+1}}\Bigg)^{\frac{1}{{p_{k+1}}}}\\
&\le \tilde{C} \bigg( \|c\|_{C^k_{p_k}(K)} +\|d c\|_{C^{k+1}_{p_{k+1}}(K)}\bigg) = \tilde{C}  \|c\|_{C^k_{{p_k},\,{p_{k+1}}}(K)} 
\end{align*}

Hence the map 
$
{\mathscr W}\colon C^k_c(K) \to S\mathscr{L}^k_{p_k,\,p_{k+1}}(K)
$
enjoys the following property \linebreak
$
\|{\mathscr W}(c)\|_{S\mathscr{L}_{{p_k},\,{p_{k+1}}}(K)} \le  \tilde{C}  \|c\|_{C^k_{{p_k},\,{p_{k+1}}}(K)} .
$
Now we can extend ${\mathscr W}$ to the closure 
$$\overline{\big(C^k_c(K)\big)}_{\|\|_{C^k_{{p_k},\,{p_{k+1}}}}} = C^k_{{p_k},\,{p_{k+1}}}(K)$$ 
by 
$$
{\mathscr W}(c) = \sum_{\sigma \in K([k])} c(\sigma){\mathscr W}(\chi_{\sigma}) = \lim_{i\to \infty} {\mathscr W}(c_i),$$ 
${as}~c_i \xrightarrow[i\to \infty]{~} c,~\{c_i\} \subset C^k_c(K),$
and  it follows that
$
{\mathscr W} \colon C^k_{{p_k},\,{p_{k+1}}}(K) \to S\mathscr{L}^k_{{p_k},\,{p_{k+1}}}(K).
$

It remains to verify that  $\mathscr I$ is a retraction. First, let $\sigma \in K[k]$  it is clear that \linebreak
${{\rm supp} ({\mathscr W }(\chi_\sigma))\cap K[k] = \{\sigma\}}$ and so
$
 \langle \mathscr I \circ {\mathscr W}(\chi_\sigma) \rangle (\sigma')= 0,~\text{if}~\sigma' \ne \sigma. 
$
Let 
$
{\mathscr W}(\chi_{\sigma}) = k!\sum_{i=0}^k (-1)^i t_i dt_0\wedge\dots\wedge dt_{i-1}\wedge dt_{i+1}\wedge \dots \wedge dt_{k},
$
so we have $t_0 = 1- \sum_{i=1}^k  t_i$ inside $\sigma$. As a result, we can write 
\begin{align*}
\frac{1}{k!} &\int_{\sigma} {\mathscr W}(\chi_{\sigma}) = \int_{\sigma}(1- \sum_{i=1}^k  t_i) dt_1\wedge\dots \wedge dt_{k}  \\
+&\int_{\sigma}\sum_{i=1}^k (-1)^i t_i d (1- \sum_{j=1}^k  t_j) \wedge\dots\wedge dt_{i-1}\wedge dt_{i+1}\wedge \dots \wedge dt_{k} 
\end{align*}
\begin{align*}
=  &\int_{\sigma}(1- \sum_{i=1}^k  t_i) dt_1\wedge\dots \wedge dt_{k}\\
+ &\int_{\sigma}\sum_{i=1}^k (-1)^{i+1} t_i d (\sum_{j=1}^k  t_j) \wedge\dots\wedge dt_{i-1}\wedge dt_{i+1}\wedge \dots \wedge dt_{k} \\
=  &\int_{\sigma}(1- \sum_{i=1}^k  t_i) dt_1\wedge\dots \wedge dt_{k}\\
+&\int_{\sigma}\sum_{i=1}^k (-1)^{i+1} t_i dt_i \wedge\dots\wedge dt_{i-1}\wedge dt_{i+1}\wedge \dots \wedge dt_{k} \\
=&\int_{\sigma}dt_1\wedge\dots \wedge dt_{k} =  \frac{\sqrt{k+1}}{k! \sqrt{2^{k}}}
\end{align*}
since 
$$dt_i \wedge\dots\wedge dt_{i-1}\wedge dt_{i+1}\wedge \dots \wedge dt_{k} = (-1)^{i-1}dt_1 \wedge\dots\wedge dt_{i-1}\wedge dt_i \wedge dt_{i+1}\wedge \dots \wedge dt_{k}.$$

Hence we have
$
 \mathscr I \circ {\mathscr W}(\chi_\sigma) = \frac{\sqrt{k+1}}{\sqrt{2^{k}}} \chi_\sigma
$
and so for $c \in C^k_{p_{k+1}}(K)$ the following holds
$
 \mathscr I \circ {\mathscr W}(c) = \frac{\sqrt{k+1}}{\sqrt{2^{k}}} c.
$
Now we can denote 
$
\tilde {\mathscr W} = \frac{\sqrt{2^{k}}}{\sqrt{k+1}} {\mathscr W},
$
that implies that the following diagram commutes
$$
\xymatrix@C =2.5em @R = 4.5em{C^k_{p_k,\,p_{k+1}}(K)\ar[dr]_{{\rm Id}_{C^k_{p_k,\,p_{k+1}}(K)}}\ar[r]^{\tilde{\mathscr W}}&~~S\mathscr{L}^k_{p_k,\,p_{k+1}}(K)\ar[d]^{\mathscr I}\\
&C^k_{p_k,\,p_{k+1}}(K)}
$$
and $\mathscr I$ is a retraction of morphism $\tilde{\mathscr W}$.
\end{proof}
{\centering
\section{On Vanishing of $\ker_\pi \mathscr I$}
}
\begin{lemma}\label{BtoI}
Let ${\Delta}$ be a simplex and $\partial \Delta$ be its boundary.  Assume that $p_k \ge p_{k+1}$.
Then any  
$ \omega \in \bigwedge^k \Omega_{\mathscr L}(\partial \Delta)$ can be extended to the whole $\Delta$ in such a way that \linebreak
$ \tilde{\omega} \in \bigwedge^k \Omega_{\mathscr L}(\Delta)$ and $\|\tilde{\omega}\|_{\Omega^{*}_{\pi}(\Delta)} \le C\|\omega\|_{\Omega^{*}_{\pi}(\partial \Delta)}$. 
\end{lemma}
\begin{proof}
Let $I = [0,\,1]$. Consider a Lipschitz  form  $ \omega \in \bigwedge^k \Omega_{\mathscr L}(\partial I^n)$. 
Our aim is to define $ \tilde{\omega} \in \bigwedge^k \Omega_{\mathscr L}(\partial I^n\times I)$ in such a manner that 
$
\tilde{\omega}\big|_{\partial I^n} = \omega
$,
{and} $\|\tilde{\omega}\|_{\Omega^{*}_{\pi}(\partial I^n\times I)} \le C \|\omega\|_{\Omega^{*}_{\pi}(\partial I^n)}$.

Define a function $f\colon I \to \mathbb{R}$ as the following 
$
t \mapsto 1-t.
$
Then we can define $\tilde{\omega}$ as the following
$
\tilde{\omega}(x,\,t) = f(t)\omega(x)
$
\begin{align*}
\|\tilde{\omega}\|^{p_k}_{\Omega^{*}_{p_k}(\partial I^n\times I)} = \int\limits_{\partial I^n\times I} |\tilde{\omega}|^{p_k} dxdt
 =  \int\limits_{I} dt\int\limits_{\partial I^n} |f(t)|^{p_k}|{\omega(x)}|^{p_k} dx \\
 = \int\limits_{0}^1 |1-t|^{p_k}dt  \int\limits_{\partial I^n} |{\omega(x)}|^{p_k} dx \le  \frac{1}{(1+p_k)^{p_k}}  \|\omega\|^{p_k}_{\Omega^{*}_{p_k}(\partial I^n)} 
 \end{align*}
$$
d\tilde{\omega} = df\wedge \omega + fd\omega = (1-t)d\omega(x) - dt\wedge\omega(x)
$$
$$
|d\tilde{\omega}| \le |1-t| |d\omega(x)| + |dt\wedge \omega(x)|
$$
$$
\|d\tilde{\omega}\|_{\Omega^{*}_{p_{k+1}}(\partial I^n\times I)} \le \|dt\wedge\omega\|_{\Omega^{*}_{p_k}(\partial I^n\times I)} + \|(1-t)d\omega\|_{\Omega^{*}_{p_{k+1}}(\partial I^n\times I)} 
$$
$$
\|(1-t)d{\omega}\|^{p_{k+1}}_{\Omega^{*}_{p_{k+1}}(\partial I^n\times I)} =
 \int\limits_{I} dt\int\limits_{\partial I^n} |1-t|^{p_{k+1}}|{d\omega(x)}|^{p_{k+1}} dx = \frac{1}{(1+p_{k+1})^{p_{k+1}}} \|d\omega\|^{p_{k+1}}_{\Omega^{*}_{p_{k+1}}(\partial I^n)} 
$$

Assume that $U$ has a finite measure and  $p_k \ge p_{k+1}$, then we can apply the H\"{o}lder's inequality
\begin{align*}
\|g\|^{p_{k+1}}_{L_{p_{k+1}}} = \|1\cdot g^{p_{k+1}}\|_{L_1}\le\|1\|_{L_{\frac{p_k}{p_k-p_{k+1}}}} \|g^{p_{k+1}}\|_{L_{\frac{p_k}{p_{k+1}}}}\\
=\Big\{ \int \limits_{U} 1 dx  \Big\}^{\frac{p_k-p_{k+1}}{p_k}} \|g\|^{p_{k+1}}_{L_{p_k}},
\end{align*}
$$
\|g\|_{L_{p_{k+1}}} \le \Big\{ \int \limits_{U} 1 dx  \Big\}^{\frac{1}{p_{k+1}}-\frac{1}{p_k}} \|g\|_{L_{p_k}}.
$$
Consequently, 
$$
\|df \wedge \omega\|_{\Omega_{p_{k+1}}(\partial I^n\times I)} = \|\omega\|_{\Omega_{p_{k+1}}(\partial I^n\times I)} 
\le \Big\{ \int \limits_{\partial I^n\times I} 1 dx  \Big\}^{\frac{1}{p_{k+1}} -\frac{1}{p_k}} \|\omega\|_{\Omega_{p_{k}}(\partial I^n \times I)}.
$$
As a result, we have 
$$
\|f\omega\|_{\Omega^{*}_{p_k}(\partial I^n \times I)}  \le \frac{1}{(1+p_k)}  \|\omega\|_{\Omega^{*}_{p_k}(\partial I^n)}, 
$$
$$
\|f d\omega\|_{\Omega_{p_{k+1}}(\partial I^n \times I)} \le \frac{1}{( 1+p_{k+1} ) } \|d\omega\|_{\Omega^{*}_{p_{k+1}}(\partial I^n)}
$$
and 
$$
\|df \wedge \omega\|_{\Omega_{p_{k+1}}(\partial I^n\times I)} 
\le  \|\omega\|_{\Omega_{p_{k}}(\partial I^n)}.
$$
$$
\|d\tilde{\omega}\|_{\Omega^{*}_{p_{k+1}}(\partial I^n\times I)} \le \|\omega\|_{\Omega^{*}_{p_{k}}(\partial I^n)}+ \|d\omega\|_{\Omega^{*}_{p_{k+1}}(\partial I^n)}
$$
\begin{align*}
\|\tilde{\omega}\|_{\Omega^{*}_{\pi}(\partial I^n\times I)} &=\|\tilde{\omega}\|_{\Omega^{*}_{p_k}(\partial I^n\times I)} + \|d\tilde{\omega}\|_{\Omega^{*}_{p_{k+1}}(\partial I^n\times I)} \\
\le&  \|\omega\|_{\Omega^{*}_{p_k}(\partial I^n)} + \|\omega\|_{\Omega^{*}_{p_{k}}(\partial I^n)}+ \|d\omega\|_{\Omega^{*}_{p_{k+1}}(\partial I^n)} 
\le 2\|\omega\|_{\Omega^{*}_{\pi}(\partial I^n)}
\end{align*}

Assume that $\Delta$ is an $n$-dimensional simplex and $ \omega \in \bigwedge^k \Omega_{\mathscr L}(\partial \Delta)$ is a Lipschitz form.

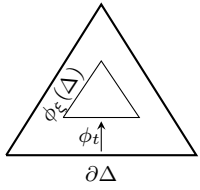
\begin{center}
\begin{tikzpicture}
\begin{scope}[scale = 0.5]
\draw[thick] (-2.5,0)--(2.5,0)--(0,4)--(-2.5,0);
\draw (-1,1)--(1,1)--(0,2.5)--(-1,1);
\draw (0,-0.5) node[font = \fontsize{8}{30}]{$\partial \Delta$};
\node[font = \fontsize{8}{30}, rotate = 57] at (-1,1.6) {$\phi_\xi(\Delta)$};
\draw[-stealth] (0,0.1) -- (0,0.9); 
\draw (-0.3, 0.5) node[font = \fontsize{8}{30}]  {$\phi_t$};
\end{scope}
\end{tikzpicture}
\end{center}

There exist  Lipschitz map $\gamma \colon \partial I^n \to  \partial \Delta $.  Indeed, there is a circumscribed sphere $S$ for the cube $\partial I^n$. 
We can use the projection $\pi_1\colon \partial I^n \to S$  along the radial direction as a bi-Lipschitz homeomorphism between $\partial I^n$ and $S$.
Then we can define a bi-Lipschitz homeomorphism between the inscribed sphere $S$ and  $\partial \Delta$  by means of  the projection  $\pi_2\colon S \to \partial \Delta$.

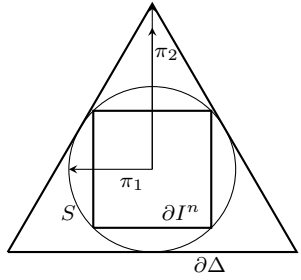
\begin{center}
\begin{tikzpicture}
\begin{scope}[scale = 3.8]
\draw[thick](-1/2,-1.732/6)--(1/2,-1.732/6)--(0,1.732/3)--(-1/2,-1.732/6);
\draw[thick](-2.45/12,-2.45/12)--(2.45/12,-2.45/12)--(2.45/12,2.45/12)--(-2.45/12,2.45/12)--(-2.45/12,-2.45/12);
\draw (0,0) circle(1.732/6);
\draw[-stealth] (0,0) -- (0,1.732/3); 
\draw[-stealth] (0,0) -- (0,1.732/3.5); 
\draw[-stealth] (0,0) -- (-1.732/6,0); 
\draw (0.1,-0.15) node[font = \fontsize{8}{30}]{$\partial I^n$};
\draw (-0.29,-0.15) node[font = \fontsize{8}{30}]{$S$};
\draw (0.2,-1.732/5.2) node[font = \fontsize{8}{30}]{$\partial \Delta$};
\draw (-0.07,-0.05) node[font = \fontsize{8}{30}]{$\pi_1$};
\draw (0.05,1.732/4.3) node[font = \fontsize{8}{30}]{$\pi_2$};
\end{scope}
\end{tikzpicture}
\end{center}

Hence, we can put $\gamma = \pi_2\pi_1$ up to the proper homothetic transformations.

Finally, we should turn to a couple of Lipschitz maps $g$, $h$ as illustrated below 

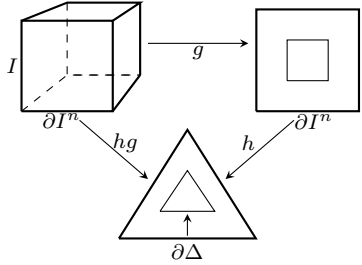
\begin{center}
\begin{tikzpicture}[scale = 0.6]
\begin{scope}[xshift = -150, scale = 0.4]
\draw[thick] (-2.5,0)--(2.5,0)--(2.5,5)--(-2.5,5)--(-2.5,0);
\draw[dashed] (0,6)--(0,2)--(4,2);
\draw[thick] (4,2)--(4,6);
\draw (-0.3,-0.5) node[font = \fontsize{8}{30}]{$\partial I^n$};
\draw (-3,2.5) node[font = \fontsize{8}{30}]{$I$};

\draw[thick] (-2.5,5)--(0,6)--(4,6)--(2.5,5);
\draw[dashed] (-2.5,0)--(0,2)--(4,2)--(2.5,0);
\draw[thick] (4,2)--(2.5,0);
\end{scope}

\draw[-stealth] (-3.5,1.5) -- (-1.3,1.5);
\draw (-2.4,1.3) node[font = \fontsize{8}{30}]{$g$};
 
\draw[-stealth] (-5,-0.2) -- (-3.5, -1.5);
\draw (-4,-0.7) node[font = \fontsize{8}{30}]{$hg$};

\begin{scope}[scale = 0.45]
\draw[thick] (-2.5,0)--(2.5,0)--(2.5,5)--(-2.5,5)--(-2.5,0);
\draw (-1,1.5)--(1,1.5)--(1,3.5)--(-1,3.5)--(-1,1.5);
\draw (0.3,-0.5) node[font = \fontsize{8}{30}]{$\partial I^n$};
\end{scope}
\draw[-stealth] (-0.3,-0.2) -- (-1.8, -1.5);
\draw (-1.3,-0.7) node[font = \fontsize{8}{30}]{$h$};

\begin{scope}[xshift = -75, yshift = -80, scale = 0.6]
\draw[thick] (-2.5,0)--(2.5,0)--(0,4)--(-2.5,0);
\draw (-1,1)--(1,1)--(0,2.5)--(-1,1);
\draw (0,-0.5) node[font = \fontsize{8}{30}]{$\partial \Delta$};
\draw[-stealth] (0,0.1) -- (0,0.9); 
\end{scope}

\end{tikzpicture}

\end{center}
We can extend $\gamma^*\omega$ as was shown above. And then put the following $\tilde{\omega} = (g^{-1}h^{-1})^* \tilde{\gamma^*\omega}$. Hence we have 
$\tilde{\omega}\big|_{\phi_{\xi}(\partial \Delta)} = 0$ and we can consider $\tilde{\omega}$ to be zero over $\phi_{\xi}(\Delta)$.
\end{proof}

\begin{lemma}
Let $K$ be an $n$-dimensional simplicial complex of bounded geometry and $K[m]$ be its $m$-dimensional skeleton.   
Then any  $ \omega \in \Omega^{*}_{\pi}(K[m])$ can be extended to the whole $K$ in such a way that
$ \tilde{\omega} \in \Omega^{*}_{\pi}(K)$ and 
$\|\tilde{\omega}\|_{\Omega^{*}_{\pi}(K)} \le C\|\omega\|_{\Omega^{*}_{\pi}(K[m])}.$
\end{lemma}
\begin{proof}
Suppose that $ \omega \in \Omega^{*}_{\pi}(K[m])$ and there is $\{\omega_i\} \subset  \bigwedge^k \Omega_{\mathscr L}(K[m])$ such that \linebreak
${\|\omega_i - \omega\|_{ \Omega^{*}_{\pi}(K[m])} \to 0}$ as $i \to \infty$.  

In the light of the previous lemma, there exists $\tilde{\omega_i} \in \Omega^{*}_{\pi}(K[m+1])$ which satisfies the following estimation 
$
\|\tilde{\omega_i}\|_{ \Omega^{*}_{\pi}(K[m+1])} \le C \|\omega_i\|_{ \Omega^{*}_{\pi}(K[m])}
$
for each $i$. It is not hard to see that $\{\tilde{\omega_i}\}$ is a Cauchy sequence since the procedure of extension is linear:
$
\|\tilde{\omega_i} - \tilde{\omega_j}\|_{ \Omega^{*}_{\pi}(K[m+1])} \to 0,~i,\,j \to \infty.
$
And so
$
\lim \|\tilde{\omega_i}\|_{ \Omega^{*}_{\pi}(K[m+1])} \le \lim\|\omega_i\|_{ \Omega^{*}_{\pi}(K[m])} = \|\omega\|_{ \Omega^{*}_{\pi}(K[m])}.
$
Denote a  limit of the sequence as the following
$
\lim\tilde{\omega_i} = \tilde{\omega}.
$
Repeating this construction for every dimension, as a result, we obtain an extension to the  whole complex.
\end{proof}
\begin{remark}
The following proposition shows that for  a non-increasing sequence $\pi$ the cohomology of $\ker \mathscr I$ is zero.
\end{remark}
\begin{proposition}\label{TrivKer}
In the following diagram in ${\sf Vec}_{\mathbb R}$
\begin{center}
$$
\xymatrix@C = 2.5em @R = 4em{
{\rm im}(d^{k})\arrow[r] \arrow@/{}^{2pc}/[rr]^0&S\mathscr{L}^{k+1}_{\pi}(K)\arrow[r]^{d^{k+1}}& S\mathscr{L}^{k+2}_{\pi}(K)\\
&{\rm ker}(\mathscr I)\arrow[u]\arrow@{-->}[lu].&
}
$$
\end{center}
the map $\ker(\mathscr I)$ factors through  ${\rm im}(d^k)$.
\end{proposition}
\begin{proof}
Let $B\subset \mathbb{R}^{k+1}$ be a $k+1$-dimensional ball.  
$$
H^i_{\text{\fontsize{5}{6}\selectfont DeRh}}(B,\,\partial B) = \begin{cases}
\mathbb{R} \text{~for~} i=k+1,\\
0 \text{~for~} i \le k.
\end{cases}
$$
Let $\xi\colon B \to \sigma$ be a bi-Lipschitz homeomorphism. Assume that 
$\omega \in Z^{k+1}_{p_{k+1}}(\sigma)$.
Thanks to the results of \cite{GT}, the de Rham cohomology of a ball coincides with $\Omega_{\pi}$-cohomology.  
It is not hard to see that $\xi^*\omega \in Z^{k+1}_{p_{k+1}}(B)$, indeed, in  accordance with \cite{GKS0}, we have:
$$
d\xi^*\omega = \xi^* d \omega = 0.
$$
And it follows that there exists $\theta \in Z^{k+1}_{p_{k+1}}(B)$  such that each $\xi^*\omega$ can be presented as
$
\xi^*\omega= c \theta + d\eta$,~ $c = \frac {\int_{B}\xi^*\omega}{\int_{B} \theta}$.
Turning back to $\sigma$, we can see the following
$$
\omega = (\xi^{-1})^*\xi^*\omega=  (\xi^{-1})^*\{c \theta\} +  (\xi^{-1})^*\{d\eta\} = c \{(\xi^{-1})^* \theta\} + d \{(\xi^{-1})^*\eta\}.
$$
Analogously, it follows that for any $\omega \in Z^i_{p_i}(\sigma)$,~$i\le k$  there exists such $\eta \in S\mathscr L^{i-1}_{p_{i-1},\,p_i}(B)$ that 
$$
\omega = d(\xi^{-1})^*\eta.
$$
That implies 
$$
H^i_{p_i,\,p_{i+1}}(\sigma,\,\partial \sigma) = \begin{cases}
\mathbb{R} \text{~for~} i=k+1,\\
0 \text{~for~} i \le k.
\end{cases}
$$

Assume that $\omega \in {\rm ker}(\mathscr I)$ is $k+1$-form.
Appealing to the above, we    can say that  \linebreak
$d(S\mathscr L^{k}_{p_{k},\,p_{k+1}}(\sigma, \partial \sigma))$ is precisely a set  of forms which have a zero integral over $\sigma$. 
Since those forms constitute a closed set, as the kernel of $\mathscr I$, so do elements of ${\rm im}(d^{k-1})$. Thereby,  the following map
$
d \colon S\mathscr L^{k}_{p_{k},\,p_{k+1}}(\sigma,\,\partial\sigma) \to {\rm im}(d^{k})
$
is an epimorphism in the category  ${\sf Ban}$,
and,  by the  Banach inverse operator theorem,
there exists a constant $C$
such that for $d \colon \eta \mapsto {\rm res_{\sigma,\,K}}\omega$ we have the estimation  $\|\eta\|_{\Omega_{\infty}(\sigma)}\le C\|{\rm res_{\sigma,\,K}}\omega\|_{\Omega_{\infty}(\sigma)}$.

To summarize, for every $\sigma \in K[k]$ there exists $\eta_{\sigma}\in \Omega_{\infty}(\sigma,\,\partial \sigma)$ such that ${\rm res}_{\sigma,\, K}\omega = d \eta_\sigma$.
Moreover, for the set of forms $\{\eta_\sigma\}_{\sigma \in K[k]}$ holds 
${\rm res}_{\sigma' \cap \sigma,\, \sigma}\eta_\sigma = {\rm res}_{\sigma' \cap \sigma,\, \sigma'}\eta_{\sigma'} = 0$ for any couple 
of simplexes $\sigma,~\sigma' \in K[k]$.  As a result, we have a ${\rm res}_{K[k],\, K}\omega = d \eta$  where ${\rm res}_{\sigma,\, K[k]} \eta = \eta_\sigma$.   

Consider the following complex
$$
\xymatrix{
0\ar[r]&\dots\ar[r]&0\ar[r]&S{\mathscr L}^{k}_{p_{k+1}}(K)\ar[r]&Z^{k+1}_{p_{k+1}}(K)\ar[r]&0\ar[r]&\dots\ar[r]&0
}
$$
Due to the above we have
$$
 \|\eta\|^{p_{k+1}}_{S\mathscr L^*_{p_{k+1}}(K[k])} = \sum_{\sigma \in K[k]}\|\eta_\sigma\|^{p_{k+1}}_{\Omega^*_{\infty}(\sigma)}
 \le  C\sum_{\sigma \in K[k]}\| {\rm res}_{\sigma,\, K}\omega \|^{p_{k+1}}_{\Omega^*_{\infty}(\sigma)} \le C\|\omega\|^{p_{k+1}}_{S\mathscr{L}^k_{p_{k+1}}(K)}.
 $$
Then we have a bounded map 
$
\gamma \colon S\mathscr{L}^{k+1}_{p_{k+1}}(K) \to S\mathscr{L}^{k}_{p_{k+1}}(K[k],K[k-1]).
$
According to the general property, we have embedding  $\ell^{p_{k+1}}(\mathbb N) \hookrightarrow \ell^{p_{k}}(\mathbb N)$ for $p_{k}>p_{k+1}$.
 So  obtained $\eta\in S\mathscr{L}^{k}_{p_{k+1},\,p_{k+1}}(K[k])$, $d \eta = \omega \in {\rm ker}(\mathscr I)$ also lies in $S\mathscr{L}^{k}_{p_{k},\,p_{k+1}}(K[k])$ for
 $p_{k}\ge p_{k+1}$.

Due to   [{\bf Lemma \ref{BtoI}}], for every simplex $\delta \in K[{k+1}]$ and all forms $\alpha \in S\mathscr{L}_{\pi}^k(\partial \delta)$, 
there is a morphism of normed spaces $s \colon S\mathscr{L}_{\pi}^k(\partial \delta)\to S\mathscr{L}_{\pi}^k(\delta)$ continuing forms from the boundary of the simplex to its interior.

In the light of previous steps, it was, in fact, established that \linebreak
$
s\gamma \colon S\mathscr{L}^{k+1}_{\pi}(K) \to S\mathscr{L}^{k}_{\pi}(K[k+1]) 
$
is a bounded map.
Let us look at 
$
\omega'  =   {\rm res}_{K[{k+1}],\, K}\omega - d(s\gamma \omega).
$
It is clear that $\omega'$ is a closed form, as well as $\omega$, and \linebreak $\omega' \in S\mathscr{L}^{k+1}_{\pi}(K[{k+1}],\,K[k])$. The restriction of $\omega'$ to every simplex 
$\sigma \in K[{k+1}]$ is an exact form
since $H^{k+1}(S\mathscr{L}^*_{\pi}(\sigma,\,\partial \sigma))  = 0$.  In other words, ${\rm im}(d^k) = {\rm ker}(d^{k+1})$
and 
$
d \colon S\mathscr{L}^{k}_{\pi}(\sigma,\,\partial\sigma) \to {\rm Im}(d^{k})
$
is an epimorphism in the category  ${\sf Ban}$.
Using  the Banach inverse operator theorem and the above argument,
we can see that there exists 
a bounded 
map 
$
\gamma^1 \colon S\mathscr{L}^{k+1}_{\pi}(K) \longrightarrow S\mathscr{L}^{k}_{\pi}(K[k+1],K[k]).
$
Then we can take 
$
s\gamma^1 \colon S\mathscr{L}^{k+1}_{\pi}(K) \longrightarrow S\mathscr{L}^{k}_{\pi}(K[k+2]) 
$
and so on. In essence, repeating this construction for every dimension, as a result, we obtain a finite composition of bounded operators which can be presented as follows
$
 {\rm ker}(\mathscr I) \to S\mathscr{L}^{k}_{\pi}(K).
$
And moreover, the map $\ker(\mathscr I)$ factors through  ${\rm im}(d^{k+1})$.
\end{proof}

\begin{remark}
For a non-monotonic sequence $\pi$, the following theorem provides us with an example which reveals that the cohomology of complex $\ker {\mathscr I}$ is not trivial. 
In this respect, we can see that  the map $\mathscr I$ loses the property being an isomorphism,
and moreover, 
$$
H^i(\Omega_\pi) \neq 0.
$$
\end{remark}
The crucial idea enclosed in the forthcoming theorem could be sublimated as follows.  It reveals a gap between complexes of Sobolev spaces which are endowed with 
counting or continuous types of measures.   The operator $\mathscr I$ precisely objectifies a chain map between two types of such complexes. Let us consider 
an equation
$$
dx = \omega
$$
where  $d\colon S{\mathscr L}^{k}_{p_{k}}(K) \to S{\mathscr L}^{k+1}_{p_{k+1}}(K)$ and $\omega \in \ker\Big(\mathscr I \colon S{\mathscr L}^{k+1}_{p_{k+1}}(K) \to C^{k+1}_{p_{k+1}}(K)\Big)$.
Due to the results of \cite{GP}, we can find a general solution of the following type
$$
x = \zeta + \ker d
$$  
for  $d\colon S{\mathscr L}^{k}_{p_{k+1}}(K') \to S{\mathscr L}^{k+1}_{p_{k+1}}(K')$. In case $p_k \ge p_{k+1}$, as shown above, $\zeta \in S{\mathscr L}^{k}_{p_{k}}(K')$  as well.
By contrast with that, assuming $p_k < p_{k+1}$, we lose the unconditional property  $\zeta \in S{\mathscr L}^{k}_{p_{k}}(K')$. 
Moreover, there exists such a family of forms $\omega$, for which, a priori, it fails.  Then a natural question arises in this context, if    
 $$\{\zeta+\ker d\} \cap S{\mathscr L}^{k}_{p_{k}}(K') = \varnothing$$
or not. In this respect, the suppose 
 $$\{\zeta+\ker d\} \cap S{\mathscr L}^{k}_{p_{k}}(K') \ne \varnothing$$
 implies a contradiction that is a kind of incarnation of the mentioned gap.  Indeed, let us assume that  
 $$\xi \in \{\zeta+\ker d\} \cap S{\mathscr L}^{k}_{p_{k}}(K').$$
 Moreover, the $p_{k+1} > p_k$ provides $\xi \in S{\mathscr L}^{k}_{p_{k+1}}(K')$. On the other hand, due to \cite{GP} we have a commutative square 
 $$
 \xymatrix{ 
 S{\mathscr L}^{k}_{p_{k+1}}(K')\ar[r]^d \ar[d]_{\mathscr I}& S{\mathscr L}^{k+1}_{p_{k+1}}(K') \ar[d]^{\mathscr I}\\
C^{k}_{p_{k+1}}(K')\ar[r]_{\partial} & C^{k+1}_{p_{k+1}}(K').
 } 
 $$ 
 We know that $\mathscr I$ always takes a $p$-summable form to a $p$-summable simplicial cochain, and moreover, for a counting measure, the differential $\partial$ preserves 
 the property of being summable with power $p$.  In this respect, if we provide $\omega \notin S{\mathscr L}^{k+1}_{p_{k}}(K')$
 in the equation
 $$
 dx=\omega,
 $$
 our hypothesis  
  $$\{\zeta+\ker d\} \cap S{\mathscr L}^{k}_{p_{k}}(K') \ne \varnothing$$
contradicts the properties of the  differential $\partial$.
 
\begin{theorem}\label{NonTriv}
Let $K$ be a noncompact metric simplicial complex with bounded geometry. Given $\pi$ such that $p_k<p_{k+1}$ ($S{\mathscr L}^{k}_{\pi}(K) = S{\mathscr L}^{k}_{p_k,\,p_{k+1}}(K)$), there does not exist an arrow $\ker {\mathscr I} \rightarrow d\big(S{\mathscr L}^{k}_{\pi}\big)$ which makes the following diagram commutative
$$\xymatrix{
d\big(S{\mathscr L}^{k}_{\pi}\big)\ar[rd]&\ker {\mathscr I}\ar[d]\ar[rd]^{0}\ar[l]|-{\times}&\\
S{\mathscr L}^{k}_{\pi}\ar[r]_{d}\ar[u]&S{\mathscr L}^{k+1}_{\pi}\ar[r]_{ {\mathscr I}}&C^{k+1}_{\pi}
}
$$
\end{theorem}
\begin{proof}
Consider a smooth function which behaves like a bump function
$$
\Psi(x) = \begin{cases}
\exp{\left(\frac{1}{|x|^2-1}\right)},~|x|<1;\\
0,~\text{otherwise}.
\end{cases}
$$
For example, it could be a kind of function presented below 
\vskip 10pt

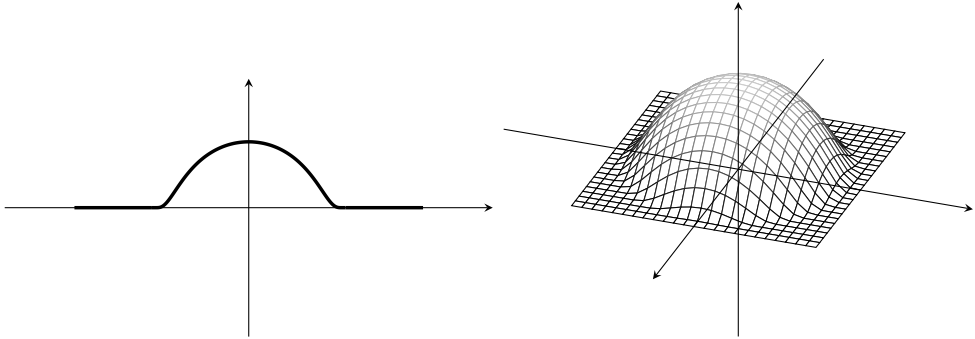
\begin{tikzpicture}
\begin{axis}[
	axis lines=middle,
	enlargelimits=true,
	ymax = 0.6,
	ymin = -0.6,
	xmax = 2.1,
	xmin = -2.1,
	xtick=\empty,
	ytick=\empty
]
\addplot[smooth,black, line width=1.3,domain=-0.999:0.999, samples=100]{e^(-1/(1-x^2))};
\addplot[smooth,black, line width=1.3] table {
	x     y
   -1   0
   -1.8   0
};
\addplot[smooth,black, line width=1.3] table {
	x     y
   1   0
   1.8   0
};
\end{axis}
\end{tikzpicture}
\begin{tikzpicture}
\begin{axis}[ 
	height=10cm,
	width=10cm,
	 axis lines=center,
    	axis on top,
	enlargelimits=true,
		view={110}{35}, 
	colormap/blackwhite,
	domain=-0.9999:0.9999,
	y domain=-0.9999:0.9999,  
	zmax = 0.6,
	zmin = -0.6,
	xmax = 1.6,
	xmin = -1.6,
	ymax = 1.6,
	ymin = -1.6,
	xtick=\empty,
	ytick=\empty,
	ztick=\empty
	]
\addplot3 [surf, fill = white, line width=0.4, unbounded coords=jump]
    { x^2+y^2<1 ? e^(-1/(1-x^2-y^2)): 0 };
\end{axis}
\end{tikzpicture}\\
\vskip 10pt
Let us take a look at  the barycentric subdivision $K'$ of  the complex $K$.  
\vskip 10pt
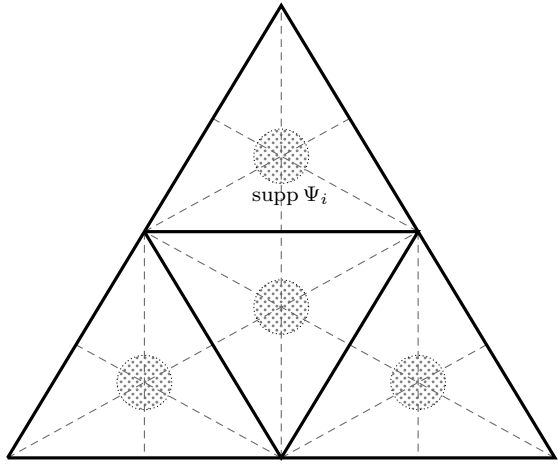
\begin{center}
\begin{tikzpicture}
\begin{scope}[scale = 1.2]
\draw [densely dotted,pattern=crosshatch dots, pattern color =  gray] (0,2.5*4/3) circle(0.3);
\draw [densely dotted,pattern=crosshatch dots, pattern color =  gray] (0,2.5*2/3) circle(0.3);
\draw [densely dotted,pattern=crosshatch dots, pattern color =  gray] (-1.5,2.5/3) circle(0.3);
\draw [densely dotted,pattern=crosshatch dots, pattern color =  gray]   (1.5,2.5/3) circle(0.3);
\draw[densely dashed,  color = black!60](0,0)--(0,5);
\draw[densely dashed,  color = black!60](-3,0)--(1.5,2.5);
\draw[densely dashed,  color = black!60](3,0)--(-1.5,2.5);
\draw[densely dashed,  color = black!60](-1.5,2.5)--(-1.5,0);
\draw[densely dashed,  color = black!60](1.5,2.5)--(1.5,0);
\draw[densely dashed,  color = black!60](-1.5,2.5)--(0.75,3.75);
\draw[densely dashed,  color = black!60](1.5,2.5)--(-0.75,3.75);
\draw[densely dashed,  color = black!60](0,0)--(-2.25,1.25);
\draw[densely dashed,  color = black!60](0,0)--(2.25,1.25);
\draw[very thick] (-3,0)--(3,0)--(0,5)--(-3,0);
\draw[very thick](-1.5,2.5)--(1.5,2.5)--(0,0)--(-1.5,2.5);
\draw (0.1, 2.9) node[font = \fontsize{8}{10}]  {$\supp\Psi_i$};
\end{scope}
\end{tikzpicture}
\end{center}
\vskip 10pt
Suppose that $\Psi(x)$ is such a bump's type  function supported in a unit ball  ${\bf B}_1 \subset \mathbb R^n$. 
Assume that $f_i$ is a bi-Lipschitz homeomorphism between ${\bf B}_1$ and the star of $K'$, which is assigned to the barycenter $b_i$ of  $\sigma_i \in K[k+1]$ in such a manner that
$b_i \mapsto 0$. Let 
$\Psi_i(x) = f_i^*\Psi(x)$ and  $0<\varepsilon<p_{k+1}-p_{k}$ then
$$
\omega_i = \Big(\frac{1}{i}\Big)^{\frac{1}{p_{k+1}-\varepsilon}}\Psi_i(x)\sum\limits_{ 0\le l_0<\dots<l_{k-1}\le n} dx_{l_0}\wedge\dots\wedge dx_{l_{k-1}}.
$$
We can derive a formula for the norm corresponding to the family of  forms defined above
$$
|\omega_i| =  \sqrt{\binom{n}{k}}\Big(\frac{1}{i}\Big)^{\frac{1}{p_{k+1}-\varepsilon}}\Psi_i(x_0,\dots,\,x_{n}).
$$
And as a result, it specifies the related sequence of maximums assigned to each simplex
$$
\max|\omega_i| =  \sqrt{\binom{n}{k}}\Big(\frac{1}{i}\Big)^{\frac{1}{p_{k+1}-\varepsilon}}.
$$
Then we can see that there exists the following limit in the Banach space $S{\mathscr L}^k_{p_k}$ 
$$
\omega = \lim_{m\to \infty}\sum^m_{i=1}\omega_i 
$$
since
$$
\Big\|\sum^m_{i=1}\omega_i \Big\|_{S{\mathscr L}^k_{p_{k+1}}} = \sqrt{\binom{n}{k}} \Bigg(\sum^m_{i=1}\Big(\frac{1}{i}\Big)^{\frac{p_{k+1}}{p_{k+1}-\varepsilon}}\Bigg)^{\frac{1}{p_{k+1}}}
$$
and 
$$
 \Bigg(\sum^\infty_{i=1}\Big(\frac{1}{i}\Big)^{\frac{p_{k+1}}{p_{k+1}-\varepsilon}}\Bigg)^{\frac{1}{p_{k+1}}}< \infty.
$$
Also, we could notice that for the obtained form, the following holds
$$
\Big\| \lim_{m\to \infty}\sum^m_{i=1}\omega_i \Big\|_{S{\mathscr L}^k_{p_{k}}}  = +\infty
$$    
since
$$
 \Bigg(\sum^\infty_{i=1}\Big(\frac{1}{i}\Big)^{\frac{p_{k}}{p_{k+1}-\varepsilon}}\Bigg)^{\frac{1}{p_{k}}} = +\infty.
$$
For instance, let us take a glance at  the rough asymptotic behaviour of such objects   
\vskip 10pt
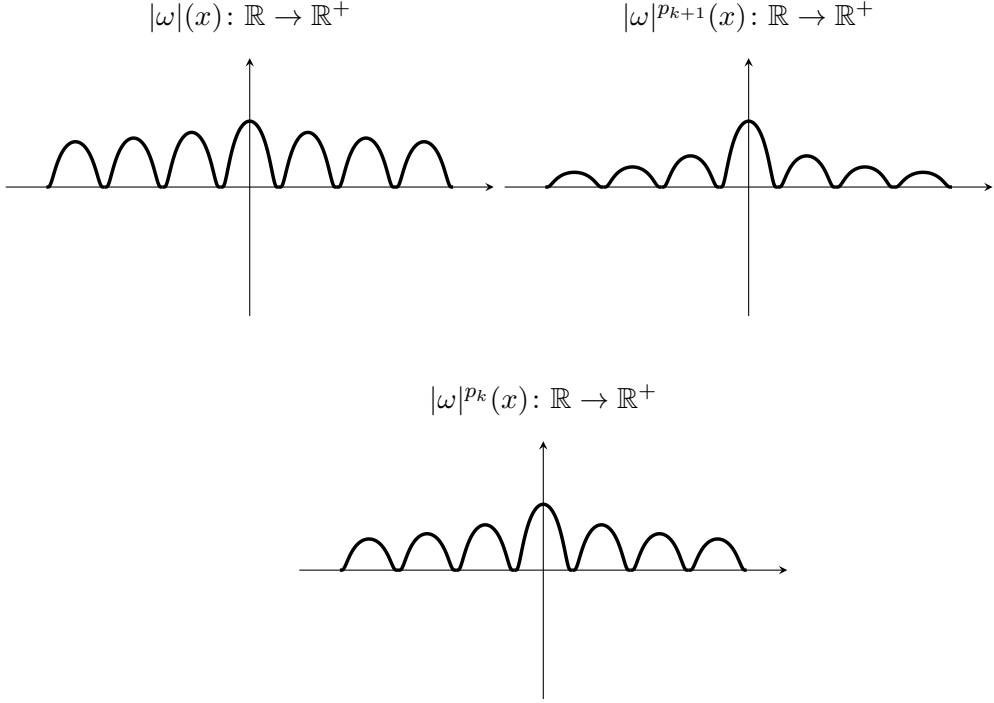
\begin{tikzpicture}
\begin{axis}[
	title = $|\omega|(x)\colon \mathbb R \to \mathbb R^+$,
	axis lines=middle,
	enlargelimits=true,
	ymax = 0.6,
	ymin = -0.6,
	xtick=\empty,
	ytick=\empty
]
\addplot[smooth,black, line width=1.3,domain=-0.999:0.999, samples=100]{e^(-1/(1-x^2))};
\addplot[smooth,black, line width=1.3, domain=-2.999:-1.001, samples=100]{(1/2)^(1/3.7)*e^(-1/(1-(x+2)^2))};
\addplot[smooth,black, line width=1.3, domain=1.001:2.999, samples=100]{(1/2)^(1/3.7)*e^(-1/(1-(x-2)^2))};
\addplot[smooth,black, line width=1.3, domain=-4.999:-3.001, samples=100]{(1/3)^(1/3.7)*e^(-1/(1-(x+4)^2))};
\addplot[smooth,black, line width=1.3, domain=3.001:4.999, samples=100]{(1/3)^(1/3.7)*e^(-1/(1-(x-4)^2))};
\addplot[smooth,black, line width=1.3, domain=-6.999:-5.001, samples=100]{(1/4)^(1/3.7)*e^(-1/(1-(x+6)^2))};
\addplot[smooth,black, line width=1.3, domain=5.001:6.999, samples=100]{(1/4)^(1/3.7)*e^(-1/(1-(x-6)^2))};
\end{axis}
\end{tikzpicture}
\begin{tikzpicture}
\begin{axis}[
	title = $|\omega|^{p_{k+1}}(x)\colon \mathbb R \to \mathbb R^+$,
	axis lines=middle,
	enlargelimits=true,
	ymax = 0.6,
	ymin = -0.6,
	xtick=\empty,
	ytick=\empty
]
\addplot[smooth,black, line width=1.3,domain=-0.999:0.999, samples=100]{e^(-1/(1-x^2))};
\addplot[smooth,black, line width=1.3, domain=-2.999:-1.001, samples=100]{(1/2)^(4/3.7)*e^(-1/(1-(x+2)^2))};
\addplot[smooth,black, line width=1.3, domain=1.001:2.999, samples=100]{(1/2)^(4/3.7)*e^(-1/(1-(x-2)^2))};
\addplot[smooth,black, line width=1.3, domain=-4.999:-3.001, samples=100]{(1/3)^(4/3.7)*e^(-1/(1-(x+4)^2))};
\addplot[smooth,black, line width=1.3, domain=3.001:4.999, samples=100]{(1/3)^(4/3.7)*e^(-1/(1-(x-4)^2))};
\addplot[smooth,black, line width=1.3, domain=-6.999:-5.001, samples=100]{(1/4)^(4/3.7)*e^(-1/(1-(x+6)^2))};
\addplot[smooth,black, line width=1.3, domain=5.001:6.999, samples=100]{(1/4)^(4/3.7)*e^(-1/(1-(x-6)^2))};
\end{axis}
\end{tikzpicture}
\\
\begin{center}
\begin{tikzpicture}
\begin{axis}[
	title = $|\omega|^{p_{k}}(x)\colon \mathbb R \to \mathbb R^+$,
	axis lines=middle,
	enlargelimits=true,
	ymax = 0.6,
	ymin = -0.6,
	xtick=\empty,
	ytick=\empty
]
\addplot[smooth,black, line width=1.3,domain=-0.999:0.999, samples=100]{e^(-1/(1-x^2))};
\addplot[smooth,black, line width=1.3, domain=-2.999:-1.001, samples=100]{(1/2)^(2/3.7)*e^(-1/(1-(x+2)^2))};
\addplot[smooth,black, line width=1.3, domain=1.001:2.999, samples=100]{(1/2)^(2/3.7)*e^(-1/(1-(x-2)^2))};
\addplot[smooth,black, line width=1.3, domain=-4.999:-3.001, samples=100]{(1/3)^(2/3.7)*e^(-1/(1-(x+4)^2))};
\addplot[smooth,black, line width=1.3, domain=3.001:4.999, samples=100]{(1/3)^(2/3.7)*e^(-1/(1-(x-4)^2))};
\addplot[smooth,black, line width=1.3, domain=-6.999:-5.001, samples=100]{(1/4)^(2/3.7)*e^(-1/(1-(x+6)^2))};
\addplot[smooth,black, line width=1.3, domain=5.001:6.999, samples=100]{(1/4)^(2/3.7)*e^(-1/(1-(x-6)^2))};
\end{axis}
\end{tikzpicture}
\end{center}

Then let us derive the explicit  expression of what $d\omega$ is.
$$
d\omega_i = \Big(\frac{1}{i}\Big)^{\frac{1}{p_{k+1}-\varepsilon}}  \sum\limits_{ 0\le l_0<\dots<l_{k}\le n} 
\Bigg(\sum\limits_{j\in\{l_0,\dots,\,l_{k}\}}  (-1)^j\frac{\partial}{\partial x_j}\Psi_i(x)\Bigg) dx_{l_0}\wedge\dots\wedge dx_{l_{k}}
$$
$$
| d\omega_i |=  \Big(\frac{1}{i}\Big)^{\frac{1}{p_{k+1}-\varepsilon}}\Bigg(\sum\limits_{ 0\le l_0<\dots<l_{k}\le n} \Bigg|\sum\limits_{j\in\{l_0,\dots,\,l_{k}\}}  (-1)^j\frac{\partial}{\partial x_j}\Psi_i(x)\Bigg|^2\Bigg)^{\frac{1}{2}}
$$
Thereby, we can sum up the above as  
$$
\max|d\omega_i| =  C\Big(\frac{1}{i}\Big)^{\frac{1}{p_{k+1}-\varepsilon}}
$$
$$
\|d\omega\|_{S{\mathscr L}^{k+1}_{p_{k+1}}} = C \Bigg(\sum^m_{i=1}\Big(\frac{1}{i}\Big)^{\frac{p_{k+1}}{p_{k+1}-\varepsilon}}\Bigg)^{\frac{1}{p_{k+1}}}<\infty
$$
$$
\|d\omega\|_{S{\mathscr L}^{k+1}_{p_{k}}} = C \Bigg(\sum^m_{i=1}\Big(\frac{1}{i}\Big)^{\frac{p_{k}}{p_{k+1}-\varepsilon}}\Bigg)^{\frac{1}{p_{k}}} =\infty
$$
Denote by $N$ the set of forms belonging to the  type specified above. 
By design,  we have
$$
\mathscr I\colon \langle N \rangle  \longrightarrow 0
$$
and 
$$
\mathscr I\colon \langle dN \rangle  \longrightarrow 0.
$$

Let us turn to the refinement $K'$. We have to verify that 
$$
{\mathscr I}\langle dN \rangle \subset C^{k+1}_{p_{k+1}}\setminus C^{k+1}_{p_{k}}.
$$
Indeed,
$$
\mathscr I \colon \sigma_i \mapsto \int\limits_{\sigma_i}d\omega_i \le  \Big(\frac{1}{i}\Big)^{\frac{1}{p_{k+1}-\varepsilon}} C,
$$
moreover, the boundness of  geometry guarantees us that  the constant $C$ does not depend on choosing a simplex. 
For example,
\begin{center}
\begin{tikzpicture}
\begin{axis}[
	domain=-0.9999:0.9999,
	axis x line=center,
  	axis y line=none,
	enlargelimits=true,
	ymax = 1.2,
	ymin = -1.2,
	xtick=\empty,
	ytick=\empty,
	extra x ticks={-1.8,0,1.8},
	extra x tick labels= {{$v_0$},{$v'$},{$v_1$}}	
]
\addplot[smooth,black, line width=1.3, samples=100]{(-2*1.6*x*e^(-1/(1-x^2))/((1-x^2)^2))};
\addplot[smooth,black, line width=1.3] table {
	x     y
   -1   0
   -1.8   0
};
\addplot[smooth,black, line width=1.3] table {
	x     y
   1   0
   1.8   0
};
\end{axis}
\end{tikzpicture}
\end{center}
$$
{\mathscr I}\colon[v_0,\,v_1] \mapsto \int^{v_1}_{v_0} \Psi'(x)dx = 0
$$
That holds since $\Psi'(x)dx \in  \ker \mathscr I$, by design. And, respectively, we have the following for  the barycentric subdivision
$$
\tilde{{\mathscr I}}\colon[v',\,v_i] \mapsto \int^{v_i}_{v'} \Psi'(x)dx = (-1)^i C\Big(\frac{1}{k}\Big)^{\frac{1}{p_{k+1}-\varepsilon}},~i = 0, 1.
$$

Let $\omega \in N$. The equation 
$$dx = d\omega$$ 
has the general solution of the following type $\{\omega+\ker d\} \subset S{\mathscr L}^{k}_{p_{k+1}}$. We need to check that 
 $$\{\omega+\ker d\} \cap S{\mathscr L}^{k}_{p_{k}} = \varnothing.$$

Assume that  there exists a closed differential form $\theta$ such that $\omega + \theta = \eta \in S{\mathscr L}^{k}_{p_{k}}\subset S{\mathscr L}^{k}_{p_{k+1}}$
and, respectively, $d\omega = d\eta$.  

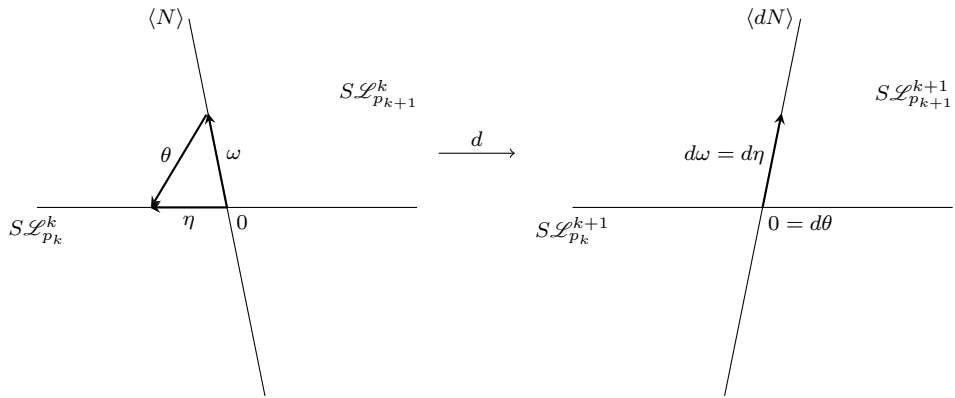
\begin{center}
\begin{tikzpicture}
\draw (-5/2,0)--(5/2,0);
\draw (-1/2,5/2)--(1/2,-5/2);
\draw[ thick,- stealth] (0,0) -- (-1-1/50,0);
\draw[ thick,- stealth] (0,0) -- (-1/4,5/4);
\draw[ thick, stealth-] (-1,0) -- (-1/4-1/50,5/4-1/50);
\draw (-5/2, -0.3) node[font = \fontsize{8}{10}]  {$S{\mathscr L}^{k}_{p_{k}}$};
\draw (2, 1.5) node[font = \fontsize{8}{10}]  {$S{\mathscr L}^{k}_{p_{k+1}}$};
\draw (-0.8, 5/2) node[font = \fontsize{8}{10}]  {$\langle N \rangle$};
\draw (0.09, 0.7) node[font = \fontsize{8}{10}]  {$\omega$};
\draw (-0.8, 0.7) node[font = \fontsize{8}{10}]  {$\theta$};
\draw (-1/2, -0.2) node[font = \fontsize{8}{10}]  {$\eta$};
\draw (0.2, -0.2) node[font = \fontsize{8}{10}]  {$0$};
\end{tikzpicture}
\begin{tikzpicture}
\draw (0, 0.8) node[font = \fontsize{8}{10}]  {$d$};
\draw (0, 5/2) node[font = \fontsize{8}{10}]  {$ $};
\draw (0, -5/2) node[font = \fontsize{8}{10}]  {$ $};
\draw[ - to] (-0.5,0.6) -- (0.5,0.6);
\end{tikzpicture}
\begin{tikzpicture}
\draw (-5/2,0)--(5/2,0);
\draw (1/2,5/2)--(-1/2,-5/2);
\draw[ thick,- stealth] (0,0) -- (1/4,5/4);
\draw (-5/2, -0.3) node[font = \fontsize{8}{10}]  {$S{\mathscr L}^{k+1}_{p_{k}}$};
\draw (2, 1.5) node[font = \fontsize{8}{10}]  {$S{\mathscr L}^{k+1}_{p_{k+1}}$};
\draw (0.1, 5/2) node[font = \fontsize{8}{10}]  {$\langle dN \rangle$};
\draw (-0.5, 0.7) node[font = \fontsize{8}{10}]  {$d\omega = d \eta$};
\draw (0.5, -0.2) node[font = \fontsize{8}{10}]  {$0 = d\theta$};
\end{tikzpicture}
\end{center}

As noted above, the maps  $\mathscr I$ and $\partial$ preserve 
 the property of being summable with power $p_k$.
So the essence of what was previously said could be expressed in the commutative diagram
in the category of vector spaces 
$$
\xymatrix@R = 4.5em @C = 2 em{
&C^k_{p_k}\arrow[r]^{\partial}\arrow@{^(->}[d]\arrow@{-->}[rd]&Z^{k+1}_{\text{\fontsize{5}{6}\selectfont s}\,p_{k}}\arrow@{_(->}[d]\\
S{\mathscr L}^{k}_{p_{k}}\arrow@{^(->}[d]\arrow^{\mathscr I}[ru]&C^k_{p_{k+1}}\arrow^{\partial}[r]&Z^{k+1}_{\text{\fontsize{5}{6}\selectfont s}\,p_{k+1}}\\
S{\mathscr L}^{k}_{p_{k+1}}\arrow^{d}[r]\arrow^{\mathscr I}[ru]&Z^{k+1}_{p_{k+1}}\arrow^{\mathscr I}[ru]
&Z^{k+1}_{\text{\fontsize{5}{6}\selectfont s}\,p_{k+1}}\setminus Z^{k+1}_{\text{\fontsize{5}{6}\selectfont s}\,p_{k}}\arrow@{^(->}[u]\\
\langle N \rangle \arrow@{-->}[u]\arrow@{-->}[r]\arrow@{_(->}[u]&\langle dN \rangle\arrow@{_(->}[u]\arrow@{-->}[ru]&
}
$$
 where $Z^{k+1}_{\text{\fontsize{5}{6}\selectfont s}\,p_{i}} = \ker \partial^{k+1}$ and $Z^{k+1}_{p_{j}} = \ker d^{k+1}$.
 Which is induced by  the commutative square
 $$
 \xymatrix{ 
 S{\mathscr L}^{k}_{p_{k+1}}(K')\ar[r]^d \ar[d]_{\mathscr I}& S{\mathscr L}^{k+1}_{p_{k+1}}(K') \ar[d]^{\mathscr I}\\
C^{k}_{p_{k+1}}(K')\ar[r]_{\partial} & C^{k+1}_{p_{k+1}}(K').
 } 
 $$ 

Let us denote $\mathscr I\colon \alpha \mapsto c_{\alpha}$.
By assumption, $\eta \in S{\mathscr L}^{k}_{p_{k}}$ and, in consequence,  $c_{\eta} \in C^{k}_{p_{k}}$. 
Since 
$\partial\colon  C^{k}_{p_{k}} \to  C^{k+1}_{p_{k}}$ is  a bounded operator, 
we have  $\partial c_{\eta} \in C^{k+1}_{p_{k}}$ as well. 
But we know that  ${\mathscr I}(d\omega) \in C^{k+1}_{p_{k+1}}\setminus C^{k+1}_{p_{k}}$.
Conversely, the  commutative diagram tells us that
$$
\partial c_{\eta} = c_{d\eta} = c_{d\omega} = \partial c_{\omega}.
$$
To be more accurate, we should notice that $\omega$ and $d\omega$ are not in $\ker \mathscr I$ for  the refinement $K'$
and also that $\partial c_{\eta} = c_{d\eta}$ holds due to the commutativity of the square  
$$
 \xymatrix{ 
 S{\mathscr L}^{k}_{p_{k+1}}(K')\ar[r]^d \ar[d]_{\mathscr I}& S{\mathscr L}^{k+1}_{p_{k+1}}(K') \ar[d]^{\mathscr I}\\
C^{k}_{p_{k+1}}(K')\ar[r]_{\partial} & C^{k+1}_{p_{k+1}}(K').
 } 
 $$ 
 and mentioned properties of $\mathscr I$ and $\partial$.
As a result,  $$\partial c_{\eta} = \partial c_{\omega}$$
  contradicts  $\partial c_{\omega} \in C^{k+1}_{p_{k+1}}\setminus C^{k+1}_{p_{k}}$ and $\partial c_{\eta} \in C^{k+1}_{p_{k}}$.
That finalises our proof. 
\end{proof}

{\centering
\section{$L_{\pi}$-Cohomology}
}

\begin{definition}
Given a sequence $\pi = \langle p_0,\,p_1,\dots,\,p_n\rangle \subset (1,\,\infty)$,  let there be a chain complex 
in the category of Banach spaces 
$$
\xymatrix{
0\ar[r]&A^0_\pi\ar[r]&\dots\ar[r]&A^{i-1}_\pi \ar[r]^{d^{i-1}}&A^i_\pi\ar[r]^{d^i}&A^{i+1}_\pi\ar[r]&\dots\ar[r]&A^n_\pi\ar[r]&0
}
$$
where $A^{i}_{\pi} = \Omega^i_\pi(K),~S\mathscr{L}^*_{\pi}(K)~\text{or}~C^i_\pi(K)$ and $d$ is a bounded operator satisfying $d^2 = 0$.
We will use the following definition for cohomology 
$$
H^{i}(A_{\pi}) = \frac{\ker\big(A^{i}_{\pi} \to A^{i+1}_{\pi}\big)}{d\big(A^{i-1}_{\pi} \big)},
$$
that specifies an additive functor 
$$
H^i \colon \mathsf{Ch(Ban)} \to  \mathsf{Vec}_{\mathbb R}.
$$
\end{definition}
\begin{theorem}
Assume $K$ is a simplicial complex of bounded geometry and  \linebreak 
$\pi = \langle p_0,\,p_1,\dots,\,p_n\rangle \subset (1,\,\infty)$ is  a non-increasing sequence such that
$\frac{1}{p_{i+1}}-\frac{1}{p_i}\le\frac{1}{n}$,
then the following cohomology groups are isomorphic in the category ${\sf Vec}_{\mathbb R}$:
$$
H^k(C^*_\pi(K)) \cong H^k(\Omega^*_\pi(K)).
$$
\end{theorem}
\begin{proof}
As shown in   [{\bf Theorem \ref{CTriang}}], we have a commutative triangle  for chain maps between complexes
 $$
\xymatrix@C = 0.2em{
&S{\mathscr L}_{\pi}\ar@{->>}[rd]^{\epsilon}&\\
\Omega_{\pi}\ar[rr]^{\mathscr R}\ar[ru]^{\rho}&&\Omega_{\pi}
}
$$
and, moreover, a homotopy equivalence 
$$
\epsilon \rho \sim 1.
$$ 
That gives
$$
H^k(\Omega^*_\pi(K)) \cong H^k(S\mathscr{L}^*_{\pi}(K)).
$$
Appealing to  the [{\bf Lemma \ref{Split}}], we can see that the chain map $\mathscr I\colon S\mathscr{L}^*_{\pi}(K) \to C^*_{\pi}(K)$ is a retraction and , as a result, we have
$$S\mathscr{L}^*_{\pi}(K) \cong C^*_{\pi}(K) \oplus {\rm ker}^k( \mathscr I).$$
Combining that with the additivity of $H^i$ and the contractibility of  the complex $\ker^* \mathscr I$, see  [{\bf Proposition \ref{TrivKer}}], we finally get  a series of isomorphisms: 
\begin{align*} 
H^k(\Omega^*_\pi(K)) \cong H^k(S\mathscr{L}^*_{\pi}(K)) &\cong H^k(C^*_{\pi}(K) \oplus \ker^* \mathscr I) \\
\cong  H^k(C^*_{\pi}(K)) \oplus& H^k(\ker ^*\mathscr I) \cong H^k(C^*_{\pi}(K))
\end{align*}
\end{proof}

\end{document}